\documentclass[11pt]{article}
\usepackage{amsmath,amssymb,textcomp,xifthen,psfrag,graphicx,color}
\usepackage{amsfonts,amsthm}
\usepackage{color}

\usepackage[utf8]{inputenc}
\usepackage[T1]{fontenc}

\usepackage{pgfplots}
\usepgfplotslibrary{external}
\usepgfplotslibrary{colorbrewer}
\usepackage{float} 
\usepackage{array}

\usepackage{multirow}

\usepackage{etoolbox}
\newcounter{rowcntr}[table]
\renewcommand{\therowcntr}{\alph{rowcntr})}

\newcolumntype{N}{>{\refstepcounter{rowcntr}\therowcntr}c}

\AtBeginEnvironment{tabular}{\setcounter{rowcntr}{0}}

\newcommand{\logLogSlopeTriangle}[6]
{

    \pgfplotsextra
    {
        \pgfkeysgetvalue{/pgfplots/xmin}{\xmin}
        \pgfkeysgetvalue{/pgfplots/xmax}{\xmax}
        \pgfkeysgetvalue{/pgfplots/ymin}{\ymin}
        \pgfkeysgetvalue{/pgfplots/ymax}{\ymax}

        \pgfmathsetmacro{\xArel}{#1}
        \pgfmathsetmacro{\yArel}{#3}
        \pgfmathsetmacro{\xBrel}{#1-#2}
        \pgfmathsetmacro{\yBrel}{\yArel}
        \pgfmathsetmacro{\xCrel}{\xArel}

        \pgfmathsetmacro{\lnxB}{\xmin*(1-(#1-#2))+\xmax*(#1-#2)} 
        \pgfmathsetmacro{\lnxA}{\xmin*(1-#1)+\xmax*#1} 
        \pgfmathsetmacro{\lnyA}{\ymin*(1-#3)+\ymax*#3} 
        \pgfmathsetmacro{\lnyC}{\lnyA+#4*(\lnxA-\lnxB)}
        \pgfmathsetmacro{\yCrel}{\lnyC-\ymin)/(\ymax-\ymin)} 

        \coordinate (A) at (rel axis cs:\xArel,\yArel);
        \coordinate (B) at (rel axis cs:\xBrel,\yCrel);
        \coordinate (C) at (rel axis cs:\xCrel,\yCrel);

        \draw[#5]   (A)--
                    (B)-- 
                    (C)-- node[pos=0.5,anchor=west] {#6}
                    cycle;
    }
}

\newcommand{\logLogSlopeTriangleBelow}[6]
{

    \pgfplotsextra
    {
        \pgfkeysgetvalue{/pgfplots/xmin}{\xmin}
        \pgfkeysgetvalue{/pgfplots/xmax}{\xmax}
        \pgfkeysgetvalue{/pgfplots/ymin}{\ymin}
        \pgfkeysgetvalue{/pgfplots/ymax}{\ymax}

        \pgfmathsetmacro{\xArel}{#1}
        \pgfmathsetmacro{\yArel}{#3}
        \pgfmathsetmacro{\xBrel}{#1-#2}
        \pgfmathsetmacro{\yBrel}{\yArel}
        \pgfmathsetmacro{\xCrel}{\xArel}

        \pgfmathsetmacro{\lnxB}{\xmin*(1-(#1-#2))+\xmax*(#1-#2)} 
        \pgfmathsetmacro{\lnxA}{\xmin*(1-#1)+\xmax*#1} 
        \pgfmathsetmacro{\lnyA}{\ymin*(1-#3)+\ymax*#3} 
        \pgfmathsetmacro{\lnyC}{\lnyA+#4*(\lnxA-\lnxB)}
        \pgfmathsetmacro{\yCrel}{\lnyC-\ymin)/(\ymax-\ymin)} 

        \coordinate (A) at (rel axis cs:\xArel,\yArel);
        \coordinate (B) at (rel axis cs:\xBrel,\yCrel);
        \coordinate (C) at (rel axis cs:\xBrel,\yArel);

        \draw[#5]   (A)--
                    (B)-- node[pos=0.5,anchor=east] {#6}
                    (C)-- 
                    cycle;
    }
}

\date{}

\newtheorem{theorem}{Theorem}
\newtheorem{lemma}[theorem]{Lemma}
\newtheorem{cor}[theorem]{Corollary}

\theoremstyle{definition} 
\newtheorem{remark}[theorem]{Remark}

\newcommand{\vdual}[2]{(#1,#2)}

\newcommand{\norm}[3][]{#1\|#2#1\|_{#3}}

\newcommand{\wilde}{\widetilde}
\newcommand{\wat}{\widehat}

\def\div{{\rm div}}

\newcommand{\osc}{{\rm osc}}
\newcommand{\ind}{{\rm ind}}

\newcommand{\R}{\ensuremath{\mathbb{R}}}

\newcommand{\Jj}{\ensuremath{\mathcal{J}}}

\newcommand{\mM}{\ensuremath{\mathcal{M}}}

\newcommand{\tT}{\ensuremath{\mathcal{T}}}

\newcommand{\pP}{\ensuremath{\mathcal{P}}}
\newcommand{\sS}{\ensuremath{\mathcal{S}}}
\newcommand{\OO}{\ensuremath{\mathcal{O}}}

\newcommand{\pphi}{{\boldsymbol\phi}}
\newcommand{\ppsi}{{\boldsymbol\psi}}

\newcommand{\ssigma}{{\boldsymbol\sigma}}

\newcommand{\xxi}{{\boldsymbol\xi}}
\newcommand{\qq}{{\boldsymbol{q}}}

\title{Space-time finite element methods for parabolic distributed optimal control problems
\thanks{Supported by Conicyt Chile through projects FONDECYT 1210579 (MK) and 1210391 (TF).}}
\author{Thomas F\"uhrer\thanks{Facultad de Matem\'{a}ticas, Pontificia Universidad Cat\'{o}lica de Chile, Santiago, Chile,
email: \texttt{tofuhrer@mat.uc.cl}}
\and 
Michael Karkulik\thanks{Departamento de Matem\'atica, Universidad T\'ecnica Federico Santa Mar\'ia, Valpara\'iso, Chile,
email: \texttt{michael.karkulik@usm.cl}}}
\begin{document}
\maketitle
\begin{abstract}
  We present a method for the numerical approximation of distributed optimal control problems constrained by 
  parabolic partial differential equations. We complement the first-order optimality condition by
  a recently developed space-time variational formulation of parabolic equations which is coercive in
  the energy norm, and a Lagrangian multiplier. Our final formulation fulfills the Babu\v{s}ka-Brezzi conditions
  on the continuous as well as discrete level, without restrictions. Consequently, we can allow for
  final-time desired states, and obtain an a-posteriori error estimator which is efficient and reliable.
  Numerical experiments confirm our theoretical findings.

\bigskip
\noindent
{\em Key words}: \\
\noindent
{\em AMS Subject Classification}:
\end{abstract}
\section{Introduction}
The problem under consideration in this manuscript is the numerical approximation of distributed optimal control of parabolic partial differential equations.
Given two target functions $u_d\in L^2(Q)$, $u_{T,d}\in L^2(\Omega)$ and an initial condition $u_0\in L^2(\Omega)$,
our goal is to find a control $f\in L^2(Q)$ which minimizes
\begin{subequations}\label{eq:copt}
\begin{align}\label{eq:copt:opt}
  J(f) := \frac{\alpha}{2} \| u - u_d \|_{L^2(Q)}^2 + \frac{\beta}{2} \| u(T) - u_{T,d} \|_{L^2(\Omega)}^2 + \frac{\lambda}{2} \| f \|_{L^2(Q)}^2,
\end{align}
where the state $u$ is the solution of the state equation
\begin{align}\label{eq:copt:heat}
  \begin{split}
  \partial_t u -\Delta u &= f\text{ in } Q,\\
  u &= 0 \text{ on } J\times \partial\Omega,\\
  u(0) &= u_0 \text{ on } \Omega.
  \end{split}
\end{align}
\end{subequations}
Above, $\Omega$ is a Lipschitz domain in $\R^d$, $J=(0,T)$ a time interval, and $Q := J\times\Omega$ the space-time cylinder.
We assume that $\lambda>0$, $\alpha,\beta \geq 0$ and $\alpha+\beta>0$.
The well-posedness of the above problem can be found in the standard reference~\cite[Ch.~2, Sec.~1.2]{Lions_71}, and can be formulated
in terms of an adjoint state, which solves a heat equation backward in time (cf. Theorem~\ref{thm:copt}).
Numerical algorithms for the solution of~\eqref{eq:copt} are computationally expensive, cf.~\cite{GoetschelMinion_2019}: iterative
schemes require the repeated evaluation of the gradient of $J$ which in turn requires the solution of the forward and the backward heat equation,
giving rise to high cost in terms of computational time. Furthermore, the transient nature of the problem requires to store information about
the discretized variables on the whole space-time domain, giving rise to a high cost in terms of memory requirement.
Another approach is to solve the discretized optimality system at once using the first-order optimality condition, the state equation, and either a Lagrangian multiplier or the adjoint state.
A discretization based on this approach naturally entails increased computational cost due to the increased size of the system that has to be solved.
Independent of the method employed to solve the optimality system, problem~\eqref{eq:copt:heat} and its backward variant need to be solved.
The most common numerical methods for this separate problem use \textit{semidiscretizations}, which lead to time-stepping schemes, cf.~\cite{Thomee}.
The advantages of time-stepping schemes is that in every time step only one spatial problem has to be solved. Furthermore, for the separate solution of~\eqref{eq:copt:heat},
they allow for optimal storage requirements as long as the final state $u(T)$ is of interest only. As mentioned above, this advantage vanishes when
dealing with transient problems such as~\eqref{eq:copt}. Furthermore, the error analysis for time-stepping schemes is only of asymptotic nature, and consequently
no quasi-optimality results for the discretized optimal control problem can be obtained either. Also, time-stepping schemes require discretizations of tensor product type,
which impedes the use of space-time adapted meshes. For distributed optimal control problems~\eqref{eq:copt} without final time desired state (i.e., $\beta=0$),
numerical methods based on time-stepping schemes can be found in the more recent works~\cite{GongHZ_12,MeidnerV_07,MeidnerV_08} and the references therein,
as well as in the monographs~\cite{BorziS_12,Troeltzsch_10}.

In order to overcome the disadvantages of time-stepping schemes, simultaneous space-time discretizations have recently
been considered for the numerical approximation of parabolic equations such as~\eqref{eq:copt:heat}.
For example, the methods proposed in~\cite{Andreev_13, Andreev_14,StevensonW_19} rely basically on the standard
well-posed variational space-time formulation of parabolic equations, cf.~\cite[Ch.~XVIII, $\mathsection$ 3]{DautrayL_92}, see also~\cite[Ch.~5]{SchwabS_09}.
This formulation is of Petrov-Galerkin type, and consequently the uniform stability in the discretization parameter
for pairs of discrete trial- and test-spaces turns out to be a major obstacle in obtaining a robust and flexible discrete method.
Hence, the mentioned works obtain uniform stability for discrete spaces based on non-uniform but, still, global time steps.
Another approach from~\cite{Steinbach_15,DevaudS_18} already allows for general simplicial space-time meshes,
but is uniformly stable only with respect to a mesh-dependent norm which is weaker than the energy norm of the heat equation if the resolution of the discrete space
goes to zero, cf.~\cite[Remark~3.5]{StevensonW_19}. Hence, this approach is not suited to control the error in $\| \cdot\|_{L^2(\Omega)}$ evaluated at a specific time,
and consequently not suited either for final time desired states ($\beta\neq 0$) in~\eqref{eq:copt:opt}.
We refer to~\cite{LangerSTY_20}, where this approach is employed to solve problem~\eqref{eq:copt} with $\beta=0$. Furthermore, we refer
to~\cite{LangerSTY_21} with energy regularization in $L^2(H^{-1}(\Omega))$, and to the recent preprint~\cite{LangerSY_22}. Furthermore,
a space-time method for the problem~\eqref{eq:copt} based on wavelet-techniques is presented in~\cite{GunzburgerK_11}.

In~\cite{FuehrerK_21,GantnerS_21}, on the other hand, a least-squares formulation of parabolic equations is developed which is coercive in the energy norm of the heat equation,
and hence naturally stable for any choice of discrete spaces. We also refer to~\cite{DieningS_22} for space-time discretization in the discontinuous Petrov-Galerkin setting,
as well as~\cite{GantnerS_22} for a least-squares formulation of the instationary Stokes equation.
In the present paper, we propose to write the the optimal control problem~\eqref{eq:copt} as a variational formulation
by using the first-order optimality conditions, including the PDE-constraint~\eqref{eq:copt:heat} using the space-time least squares formulation from~\cite{FuehrerK_21,GantnerS_21},
complemented with a Lagrangian multiplier. We show that this saddle-point system fulfills the Babu\v{s}ka-Brezzi conditions on the continuous \textit{as well as} the discrete level,
without any restrictions on discrete spaces used. The use of least-squares techniques for the numerical solution of optimal control problem is treated already comprehensively
in~\cite[Sec.~11]{BochevG_09}, and the different approaches are compared with respect to their properties in~\cite[Table~11.1]{BochevG_09}. The method considered in the present
paper corresponds to approach no. $4$, which~\cite{BochevG_09} observe to be uniformly stable for any choice of discrete spaces, in accordance with our results. Furthermore,
the system matrix is symmetric and optimal error estimates can be obtained. This will also be stated explicitely in the present paper. We also obtain
an efficient and reliable a-posteriori error estimator which allows us to steer an adaptive algorithm.

\section{Variational formulation of the optimal control problem}
\subsection{Preliminaries}
For a bounded (spatial) Lipschitz domain $\Omega\subset\R^d$ we consider the standard Lebesgue and Sobolev spaces
$L^2(\Omega)$ and $H^k(\Omega)$ for $k\geq 1$ with the standard norms. The space $H^1_0(\Omega)$ consists of all $H^1(\Omega)$ functions
with vanishing trace on the boundary $\partial\Omega$.
We define $H^{-1}(\Omega):= H^1_0(\Omega)'$ as topological dual
with respect to the extended $L^2(\Omega)$ scalar product $\vdual{\cdot}{\cdot}_\Omega$.
For the time interval $J=(0,T)$ and a Banach space $X$ we will use the space $L^2(X)$ of functions
$f:J\rightarrow X$ which are strongly measurable with respect to the Lebesgue measure $ds$ on $\R$ and
\begin{align*}
  \| f \|_{L^2(X)}^2 := \int_J \| f(s) \|_X^2\,ds < \infty.
\end{align*}
A function $f\in L^2(X)$ is said to have a weak time-derivative $f'\in L^2(X)$, if
\begin{align*}
  \int_J f'(s)\cdot\varphi(s)\,ds = -\int_J f(s)\cdot \varphi'(s)\,ds\quad\text{ for all test functions }\varphi \in C^\infty_0(J).
\end{align*}
We then define the Sobolev-Bochner space $H^k(X)$ of functions in $L^2(X)$ whose weak derivatives $f^{(\alpha)}$ of all orders $|\alpha|\leq k$
exist, endowed with the norm
\begin{align*}
  \| f \|_{H^k(X)}^2 := \sum_{|\alpha|\leq k} \| f^{(\alpha)} \|_{L^2(X)}^2.
\end{align*}
If we denote by $C(\overline J;X)$ the space of continuous functions $f:\overline{J}\rightarrow X$
endowed with the natural norm, then we have the following time trace theorem,
\begin{align}\label{eq:timetrace}
  L^2(H^1_0(\Omega)) \cap H^1(H^{-1}(\Omega)) &\hookrightarrow C(\overline J; L^2(\Omega)),
\end{align}
cf.~\cite[Ch.~5.9.2, Thm.~3]{Evans}. If $\partial\Omega$ is smooth, there even holds for a nonnegative integer $m$
\begin{align}\label{eq:timetrace:2}
  L^2(H^{m+2}(\Omega)) \cap H^1(H^m(\Omega)) &\hookrightarrow C(\overline J;H^{m+1}(\Omega)),
\end{align}
cf.~\cite[Ch.~5.9.2, Thm.~4]{Evans}. The proof of this result in the last reference uses extension operators requiring $\partial\Omega$
to be smooth. This smoothness assumption can be circumvented by using Stein's extension operator, cf.~\cite[Ch.~2, Thm.~5]{Stein_70}.
If $u \in L^2(H^2(\Omega)\cap H^1_0(\Omega)) \cap H^1(L^2(\Omega))$, then we can apply the trace operator
$\gamma: H^1(\Omega)\rightarrow H^{1/2}(\partial\Omega)$ and see $0 = \gamma u \in L^2(H^{1/2}(\partial\Omega))$, and due to~\eqref{eq:timetrace:2}
also $\gamma u\in C(\overline J;H^{1/2}(\partial\Omega))$.
We conclude
\begin{align}\label{eq:timetrace:3}
  L^2(H^2(\Omega)\cap H^1_0(\Omega)) \cap H^1(L^2(\Omega)) &\hookrightarrow C(\overline J;H^1_0(\Omega)).
\end{align}
We will also need the following result. For a proof see, e.g.,~\cite[Lem.~2]{FuehrerK_21}.
\begin{lemma}\label{lem:schwarz}
  Let $L:X\rightarrow Y$ be a linear and bounded operator between two Banach spaces $X$, $Y$ and $u\in H^1(X)$. Then, $Lu \in H^1(Y)$ and $(Lu)' = L(u')$ holds.
\end{lemma}
We can then show the following result.
\begin{cor}\label{cor:schwarz}
  The following identity holds,
  \begin{align*}
    L^2(H^1_0(\Omega))\cap H^1(H^1(\Omega)) = H^1(H^1_0(\Omega)).
  \end{align*}
\end{cor}
\begin{proof}
  It suffices to show that $u\in L^2(H^1_0(\Omega))\cap H^1(H^1(\Omega))$ implies $u\in H^1(H^1_0(\Omega))$.
  To that end, denote by $\gamma:H^1(\Omega)\rightarrow L^2(\partial\Omega)$ the spatial trace operator.
  Then, $0=\gamma u \in L^2(L^2(\partial\Omega))$. According to Lemma~\ref{lem:schwarz}, $\gamma(u) \in H^1(L^2(\partial\Omega))$, and
  $\gamma(u') = \gamma(u)' = 0$, hence $u\in H^1(H^1_0(\Omega))$.
\end{proof}
The following regularity result is found in~\cite[Ch.~7.1, Thm.~5, Thm.~6]{Evans}.
\begin{lemma}\label{lem:par:reg}
  Suppose that
  \begin{align*}
    \partial_t v -\Delta v &= g\text{ in } Q,\\
    v &= 0 \text{ on } J\times \partial\Omega,\\
    v(0) &= v_0 \text{ on } \Omega.
  \end{align*}
  \begin{enumerate}
    \item If $g\in L^2(L^2(\Omega))$ and $v_0\in H^1_0(\Omega)$, then $u\in L^2(H^2(\Omega))\cap H^1(L^2(\Omega))\cap C(\overline J;H^1_0(\Omega))$.
    \item If $g\in L^2(H^2(\Omega))\cap H^1(L^2(\Omega))$, $v_0\in H^1_0(\Omega)\cap H^3(\Omega)$, and
      $g(0)-\Delta v_0\in H^1_0(\Omega)$, then
      $u \in L^2(H^4(\Omega))\cap H^1(H^2(\Omega)\cap H^1_0(\Omega))\cap H^2(L^2(\Omega))$.
  \end{enumerate}
\end{lemma}
\subsection{Space-time variational formulation of the heat equation}\label{sec:heat}
In this section we briefly introduce the coercive formulation of the heat equation~\eqref{eq:copt:heat} which we
considered in~\cite{FuehrerK_21}, see also~\cite{GantnerS_21} for a generalization.
Define the Hilbert space
\begin{align*}
  U := \left\{ (v,\ppsi)\mid v \in L^2(H^1_0(\Omega)), \ppsi\in L^2(Q), \partial_t v -\div\,\ppsi\in L^2(Q) \right\}.
\end{align*}
As was shown in~\cite[Prop.~2.1]{GantnerS_21}, $(v,\ppsi)\in U$ implies $\partial_tv\in L^2(H^{-1}(\Omega))$ and
\begin{align}\label{eq:FK:GS}
  \| v \|_{L^2(H^1_0(\Omega))\cap H^{1}(H^{-1}(\Omega))} \lesssim \| (v,\ppsi) \|_{U} \quad\text{for all } (v,\ppsi)\in U.
\end{align}
Define the bilinear form $b: U\times U\rightarrow\R$ by
\begin{align*}
  b(u,\ssigma;v,\ppsi) := \vdual{\nabla u-\ssigma}{\nabla v-\ppsi}_{Q} +
  \vdual{\partial_tu-\div\,\ssigma}{\partial_t v-\div\,\ppsi}_{Q}
  +\vdual{u(0)}{v(0)}_{\Omega}
\end{align*}
The following result is due to~\cite{FuehrerK_21,GantnerS_21}.
\begin{lemma}\label{lem:b:ell}
  The bilinear form $b:U\times U\rightarrow\R$ is coercive and bounded on $U$.
\end{lemma}
Note that the weak solution $u$ of~\eqref{eq:copt:heat} fulfills $(u,\nabla u)\in U$ and solves
\begin{align*}
  b(u,\nabla u;v,\ppsi) = \vdual{f}{\partial_tv-\div\,\ppsi}_{Q} + \vdual{u_0}{v(0)}_{\Omega} \quad \text{ for all }(v,\ppsi)\in U.
\end{align*}
The right-hand side in the preceding formulation is bounded due to Cauchy--Schwarz, the time-trace theorem~\eqref{eq:timetrace}, and~\eqref{eq:FK:GS},
\begin{align*}
  |\vdual{f}{\partial_tv-\div\,\ppsi}_Q + \vdual{u_0}{v(0)}_\Omega\| &\leq \|f\|_Q\|\partial_tv-\div\,\ppsi\|_Q +\|u_0\|_\Omega\|v(0)\|_\Omega 
  \\&\lesssim (\|f\|_Q+\|u_0\|_\Omega)\|(v,\ppsi)\|_U.
\end{align*}
Together with Lemma~\ref{lem:b:ell}, the Lax--Milgram lemma then implies that $(u,\nabla u)$ is also the unique solution of this variational formulation.
\subsection{Space-time variational formulation of the optimal control problem}
Now we compute the first-order optimality conditions for the functional $J$ from~\eqref{eq:copt:opt} and include the state equation
using the bilinear form $b$ defined in the previous section. Hence, we are looking for
$f\in L^2(Q)$ and $(u,\qq)\in U$ such that
\begin{align*}
  \alpha\vdual{u}{w}_{Q} + \beta \vdual{u(T)}{w(T)}_\Omega + \lambda \vdual{f}{g}_{Q}&= \alpha\vdual{u_d}{w}_{Q} + \beta\vdual{u_{T,d}}{w(T)}_\Omega\\
  b(u,\qq;v,\ppsi)-\vdual{f}{\partial_tv-\div\,\ppsi}_Q &= \vdual{u_0}{v(0)}_\Omega
\end{align*}
for all $g\in L^2(Q)$ and $(v,\ppsi)\in U$, where $w$ solves
\begin{align*}
  \partial_t w -\Delta w &= g\text{ in } Q,\\
  w &= 0 \text{ on } J\times \partial\Omega,\\
  w(0) &= 0 \text{ on } \Omega.
\end{align*}
Using a Lagrangian multiplier, we arrive at the problem to
find $f\in L^2(Q)$, $(u,\qq),(y,\xxi)\in U$, such that
\begin{align*}
  \alpha\vdual{u}{w}_{Q} + \beta\vdual{u(T)}{w(T)}_\Omega + \lambda (f,g)_{Q} + b(w,\pphi;y,\xxi)
  - \vdual{g}{\partial_t y - \div\,\xxi}_{Q} &= \alpha\vdual{u_d}{w}_{Q} + \beta\vdual{u_{T,d}}{w(T)}_\Omega\\
  b(u,\qq;v,\ppsi)-\vdual{f}{\partial_tv-\div\,\ppsi} &= \vdual{u_0}{v(0)}_\Omega,
\end{align*}
for all $g\in L^2(Q)$, $(w,\pphi),(v,\ppsi)\in U$.
Let us write this problem in a compact form by introducing
\begin{align*}
  V := L^2(Q)\times U
\end{align*}
and bilinear forms $a:V^2\rightarrow\R$ and
$\wilde b:V\times U\rightarrow\R$,
\begin{align*}
  a(f,u,\qq;g,w,\pphi) &:= \alpha\vdual{u}{w}_{Q} + \beta\vdual{u(T)}{w(T)}_\Omega + \lambda (f,g)_{Q}\\
  \wilde b(f,u,\qq;v,\ppsi) &:= b(u,\qq;v,\ppsi)-\vdual{f}{\partial_tv-\div\,\ppsi}.
\end{align*}
Eventually, we arrive at the following formulation: find $(f,u,\qq)\in V$ and $(y,\xxi)\in U$ such that
\begin{align}\label{eq:copt:var}
  \begin{split}
    a(f,u,\qq;g,w,\pphi) + \wilde b(g,w,\pphi;y,\xxi) &= \alpha\vdual{u_d}{w}_{Q} + \beta\vdual{u_{T,d}}{w(T)}_\Omega\\
    \wilde b(f,u,\qq;v,\ppsi) &= \vdual{u_0}{v(0)}_\Omega,
  \end{split}
\end{align}
for all $(g,w,\pphi)\in V$ and $(v,\ppsi)\in U$.
\begin{theorem}\label{thm:copt:wp}
  Problem~\eqref{eq:copt:var} is well-posed: there is a unique solution $(f,u,\qq)\in V$ and $(y,\xxi)\in U$, and
  \begin{align*}
    \| (f,u,\qq) \|_V + \|(y,\xxi)\|_U \lesssim \alpha \| u_d \|_{L^2(Q)} + \beta \| u_{T,d} \|_{L^2(\Omega)} + \| u_0 \|_{L^2(\Omega)}.
  \end{align*}
\end{theorem}
\begin{proof}
  We will verify the conditions of Brezzi's theory, cf.~\cite[Thm~4.2.3]{BoffiBF_13}.
  The right-hand side in~\eqref{eq:copt:var} is a linear functional on $V\times U$, and due to the Cauchy--Schwarz inequality and the time-trace theorem~\eqref{eq:timetrace} its norm
  is bounded by
  \begin{align*}
    \alpha \| u_d \|_{L^2(Q)} + \beta \| u_{T,d} \|_{L^2(\Omega)} + \| u_0 \|_{L^2(\Omega)}.
  \end{align*}
  The Cauchy--Schwarz inequality and~\eqref{eq:timetrace} also show that the bilinear forms $a$ and $\wilde b$ are bounded.
  Now let $(v,\ppsi)\in U$ be arbitrary.
  Then, choose $(f,u,\qq) = (0,v,\ppsi)$ and note that due to the coercivity of $b$ from Lemma~\ref{lem:b:ell}
  \begin{align*}
    \wilde b(f,u,\qq;v,\ppsi) = b(v,\ppsi;v,\ppsi)\gtrsim \norm{(v,\ppsi)}{U}^2 = \norm{(v,\ppsi)}{U}\norm{(f,u,\qq)}{V}.
  \end{align*}
  This shows the continuous inf-sup condition for $\wilde b$. Next let $(f,u,\qq)\in\ker \wilde b$, that is,
  \begin{align*}
    b(u,\qq;v,\ppsi)=\vdual{f}{\partial_tv-\div\,\ppsi} \quad \text{ for all } (v,\ppsi)\in U.
  \end{align*}
  Then, due to Lemma~\ref{lem:b:ell} and boundedness of the right-hand side, we have
  \begin{align*}
    \norm{(u,\qq)}{U} \lesssim \norm{f}{L^2(Q)},
  \end{align*}
  and this yields
  \begin{align*}
    \norm{(u,\qq)}{U}^2 + \norm{f}{L^2(Q)}^2 \lesssim \norm{f}{L^2(Q)}^2 \lesssim a(f,u,\qq;f,u,\qq),
  \end{align*}
  which proves the coercivity of $a$ on $\ker\wilde b$.
\end{proof}
Next, we will derive some regularity results. To this end, we will characterize the
Lagrangian multiplier $(y,\xxi)$ in~\eqref{eq:copt:var} in terms of the so-called \textit{adjoint state}.
The following result can be found in~\cite[Ch.~3, Sec.~2.3]{Lions_71}.
\begin{theorem}\label{thm:copt}
  The parabolic optimal control problem~\eqref{eq:copt} has a unique solution $f\in L^2(Q)$, $u\in L^2(H^1_0(\Omega))\cap H^1(H^{-1}(\Omega))$,
  determined uniquely by~\eqref{eq:copt:heat}, the adjoint state $p \in L^2(H^1_0(\Omega))\cap H^1(H^{-1}(\Omega))$ given by
  \begin{align}\label{eq:copt:transpose}
    \begin{split}
      -\partial_t p -\Delta p &= \alpha \left( u-u_d \right)\text{ in } Q,\\
      p &= 0 \text{ on } J\times \partial\Omega,\\
      p(T) &= \beta \left( u(T)-u_{T,d} \right) \text{ on } \Omega,
    \end{split}
  \end{align}
  and $p = -\lambda f$.
\end{theorem}
\begin{remark}
  We note that~\eqref{eq:copt:transpose} is basically a heat equation, which can be seen using the change of variables $t\mapsto T-t$.
  Hence, without further ado, the regularity results of Lemma~\ref{lem:par:reg} will be applied also in the case of the adjoint state $p$.
\end{remark}
\begin{lemma}\label{lem:lagmult}
  Let $(f,u,\qq)\in V$, $(y,\xxi)\in U$ be the solution of~\eqref{eq:copt:var} and $p$ the adjoint state given by Theorem~\ref{thm:copt}.
  Then, $\qq=\nabla u$,
  \begin{subequations}\label{eq:lag:adj}
  \begin{align}
    \begin{split}
      \partial_t y -\Delta y &= -p+\Delta p\text{ in } Q,\\
      y &= 0 \text{ on } J\times \partial\Omega,\\
      y(0) &= -p(0) \text{ on } \Omega,
    \end{split}
  \end{align}
  and
  \begin{align}
    \xxi = \nabla y + \nabla p.
  \end{align}
  \end{subequations}
\end{lemma}
\begin{proof}
  We note that from~\eqref{eq:copt:var} it follows that the Lagrangian multiplier $(y,\xxi)\in U$ satisfies
  \begin{align}\label{lem:lagmult:eq1}
    b(w,\pphi;y,\xxi) = \alpha\vdual{u_d-u}{w}_{Q} + \beta\vdual{u_{T,d}-u(T)}{w(T)}_\Omega \qquad \text{ for all } (w,\pphi)\in U,
  \end{align}
  and due to coercivity of $b$ it is also the unique solution of~\eqref{lem:lagmult:eq1}. Hence, it suffices to show that a solution
  $(y,\xxi)$ of~\eqref{eq:lag:adj} also solves~\eqref{lem:lagmult:eq1}.
  This can be seen easily by integration by parts in space and time. If $(y,\xxi)$ solves~\eqref{eq:lag:adj}, then
  $\partial_ty-\div\,\xxi = -p$, and
  \begin{align*}
    b(w,\pphi;y,\xxi) &= 
    \vdual{\partial_ty-\div\,\xxi}{\partial_tw-\div\,\pphi}_Q + \vdual{\nabla y-\xxi}{\nabla w-\pphi}_Q + (w(0),y(0))_\Omega\\
    &=
    \vdual{-p}{\partial_t w-\div\,\pphi}_Q - \vdual{w(0)}{p(0)}_\Omega - \vdual{\nabla p}{\nabla w-\pphi}_Q\\
    & = \vdual{\partial_t p}{w}_Q - \vdual{\nabla p}{\nabla w}_Q + \beta\vdual{u_{T,d}-u(T)}{w(T)}_\Omega\\
    &= \alpha\vdual{u_d-u}{w}_Q + \beta\vdual{u_{T,d}-u(T)}{w(T)}_\Omega,
  \end{align*}
  which shows~\eqref{lem:lagmult:eq1}.
\end{proof}
We can now show the following regularity results.
\begin{lemma}\label{lem:copt:regularity}
  Let $\Omega\subset\R^d$ with $\partial\Omega$ be smooth, or, particularly, $\Omega\subset\R$ an interval.
  Suppose that $u_0\in H^1_0(\Omega)\cap H^3(\Omega)$ with $\Delta u_0\in H^1_0(\Omega)$, and
  $u_d\in L^2(H^2(\Omega))\cap H^1(L^2(\Omega))$ with $u_d(T)\in H^1_0(\Omega)$,
  and $u_{T,d}\in H^1_0(\Omega)$ with $\Delta u_{T,d}\in H^1_0(\Omega)$.
  Let $(f,u,\qq)\in U$, $(y,\xxi)\in V$ be the solution of~\eqref{eq:copt:var}, cf.~Thm.~\ref{thm:copt:wp},
  and $p\in L^2(H^1_0(\Omega))\cap H^1(H^{-1}(\Omega))$ the adjoint state, cf.~Thm.~\ref{thm:copt}.
  Then
  \begin{align}\label{lem:copt:regularity:eq1}
    f,u,p &\in L^2(H^4(\Omega))\cap H^1(H^2(\Omega))\cap H^2(L^2(\Omega)).
  \end{align}
  Furthermore,~\eqref{lem:copt:regularity:eq1} implies
  \begin{align*}
    y &\in L^2(H^4(\Omega))\cap H^1(H^2(\Omega))\cap H^2(L^2(\Omega)).
  \end{align*}
\end{lemma}
\begin{proof}
  If $(f,u,\qq)\in U$, $(y,\xxi)\in V$ solves~\eqref{eq:copt:var}, then in particular
  \begin{align*}
    \wilde b(f,u,\qq;v,\ppsi) = \vdual{u_0}{v(0)} \qquad\text{ for all } (v,\ppsi)\in V,
  \end{align*}
  which means that $u$ fulfills~\eqref{eq:copt:heat}. As $u_0\in H^1_0(\Omega)$ and $f\in L^2(L^2(\Omega))$, we obtain
  with Lemma~\ref{lem:par:reg}
  \begin{align*}
    u \in L^2(H^2(\Omega))\cap H^1(L^2(\Omega)) \cap C(\overline J;H^1_0(\Omega)).
  \end{align*}
  Next, $p$ fulfills~\eqref{eq:copt:transpose}. As $u-u_d\in L^2(L^2(\Omega))$ and, by the aforegoing result and hypothesis,
  $u(T)-u_{T,d}\in H^1_0(\Omega)$,
  we obtain with Lemma~\ref{lem:par:reg}
  \begin{align*}
    p,f \in L^2(H^2(\Omega))\cap H^1(L^2(\Omega)) \cap C(\overline J;H^1_0(\Omega)),
  \end{align*}
  where we also used $p = -\lambda f$. Next, we will bootstrap the obtained regularities. By hypothesis,
  $u_0\in H^1_0(\Omega)\cap H^3(\Omega)$ and $\Delta u_0\in H^1_0(\Omega)$. From the results already obtained,
  $f\in L^2(H^2(\Omega))\cap H^1(L^2(\Omega))$ and $f(0)\in H^1_0(\Omega)$.
  In particular, $f(0)-\Delta u_0 \in H^1_0(\Omega)$, and hence Lemma~\ref{lem:par:reg} gives
  \begin{align*}
    u &\in L^2(H^4(\Omega))\cap H^1(H^2(\Omega)\cap H^1_0(\Omega))\cap H^2(L^2(\Omega)).
  \end{align*}
  Next, by hypothesis $u_d\in L^2(H^2(\Omega))\cap H^1(L^2(\Omega))$ as well as $u_d(T),\Delta u_{T,d}\in H^1_0(\Omega)$.
  From the results already obtained, $f(T) \in H^1_0(\Omega)$. Furthermore, $u'\in L^2(H^2(\Omega)\cap H^1_0(\Omega))\cap H^1(L^2(\Omega))$.
  Hence, due to Corollary~\ref{cor:schwarz}, $u'\in C(\overline J;H^1_0(\Omega))$. In particular, $-\Delta u(T) = f(T)-u'(T)\in H^1_0(\Omega)$.
  We conclude that $u(T)-u_d(T) - \Delta(u(T) - u_{T,d})\in H^1_0(\Omega)$, and together with
  $u-u_d\in L^2(H^2(\Omega))\cap H^1(L^2(\Omega))$ we obtain, using again Lemma~\ref{lem:par:reg},
  \begin{align*}
    p &\in L^2(H^4(\Omega))\cap H^1(H^2(\Omega))\cap H^2(L^2(\Omega)).
  \end{align*}
  Finally, note that $(y,\xxi)$ solves~\eqref{eq:lag:adj}. Given that $-p+\Delta p\in L^2(H^2(\Omega))\cap H^1(L^2(\Omega))$ as well as
  $-p(0) + \Delta p(0) - \Delta p(0) = -p(0)\in H^1_0(\Omega)$, we may again apply Lemma~\ref{lem:par:reg} to obtain
  \begin{align*}
    y &\in L^2(H^4(\Omega))\cap H^1(H^2(\Omega))\cap H^2(L^2(\Omega)).
  \end{align*}
\end{proof}
\section{A space-time finite element method}
Now let $V_h\subset V$ and $U_h\subset U$ be closed subspaces, and consider the problem to find
$(f_h,u_h,\qq_h)\in V_h$ and $(y_h,\xxi_h)\in U_h$ such that
\begin{align}\label{eq:copt:disc}
  \begin{split}
    a(f_h,u_h,\qq_h;g_h,w_h,\pphi_h) + \wilde b(g_h,w_h,\pphi_h;y_h,\xxi_h) &= \alpha\vdual{u_d}{w_h}_{Q} + \beta\vdual{u_{T,d}}{w_h(T)}_\Omega\\
    \wilde b(f_h,u_h,\qq_h;v_h,\ppsi_h) &= \vdual{u_0}{v_h(0)}_\Omega,
  \end{split}
\end{align}
for all $(g_h,w_h,\pphi_h)\in V_h$ and $(v_h,\ppsi_h)\in U_h$.
\begin{theorem}\label{thm:copt:disc:wp}
  Problem~\eqref{eq:copt:disc} is well-posed: there is a unique solution $(f_h,u_h,\qq_h)\in V_h$ and $(y_h,\xxi_h)\in U_h$, and
  \begin{align*}
    \| (f_h,u_h,\qq_h) \|_V + \|(y_h,\xxi_h)\|_U \lesssim \alpha \| u_d \|_{L^2(Q)} + \beta \| u_{T,d} \|_{L^2(\Omega)} + \| u_0 \|_{L^2(\Omega)}.
  \end{align*}
  Furthermore, with $(f,u,\qq)\in V$ and $(y,\xxi)\in U$ being the solution of Theorem~\ref{thm:copt:wp}, we have the quasi-optimality
  \begin{align*}
    &\| (f,u,\qq) - (f_h,u_h,\qq_h) \|_V + \| (y,\xxi) - (y_h,\xxi_h) \|_U\\
    &\qquad\qquad\qquad\qquad\lesssim \inf_{\substack{(\wilde f_h,\wilde u_h, \wilde\qq_h)\in V_h\\(\wilde y_h,\wilde \xxi_h)\in U_h}}
    \| (f,u,\qq) - (\wilde f_h,\wilde u_h,\wilde \qq_h) \|_V + \| (y,\xxi) - (\wilde y_h,\wilde \xxi_h) \|_U
  \end{align*}
\end{theorem}
\begin{proof}
  Again we verify the conditions of the Babu\v{s}ka-Brezzi theory, but on the discrete level, cf.~\cite[Thms.~5.2.1 and 5.2.2]{BoffiBF_13}.
  As the bilinear form $b$ is coercive, we can employ the same arguments as in the proof of Theorem~\ref{thm:copt:wp}.
  The quasi-optimality estimate is a standard consecuence in coercive Galerkin settings.
\end{proof}
The least-squares structure for $b$ allows us to develop an a-posteriori error estimator which is reliable and efficient.
\begin{theorem}\label{thm:apost}
  Let $V_h\subset V$ and $U_h\subset U$ be closed subspaces.
  Let $(f,u,\qq)\in V$, $(y,\xxi)\in U$ be the solution of~\eqref{eq:copt:var}
  and $(f_h,u_h,\qq_h)\in V_h$ and $(y_h,\xxi_h)\in U_h$ be the solution of~\eqref{eq:copt:disc}.
  Let $\wilde p$ be given by
  \begin{align*}
      -\partial_t \wilde p -\Delta \wilde p &= \alpha \left( u_h-u_d \right)\text{ in } Q,\\
      \wilde p &= 0 \text{ on } J\times \partial\Omega,\\
      \wilde p(T) &= \beta \left( u_h(T)-u_{T,d} \right) \text{ on } \Omega.
  \end{align*}
  Define the a posteriori error estimator
  \begin{align*}
    \eta^2 &:= \| \lambda^{-1}(\partial_ty_h-\div\,\xxi_h) - f_h \|_Q^2 + \| \qq_h-\nabla u_h \|_{Q}^2
    + \| \partial_t u_h - \div\,\qq_h-f_h \|_Q^2 + \| u_h(0) - u_0 \|_\Omega^2\\
    &\qquad+ \| \xxi_h-\nabla y_h - \nabla\wilde p \|_{Q}^2 + \| \partial_t y_h - \div\,\xxi_h+\wilde p \|_Q^2.
  \end{align*}
  The estimator is reliable and efficient, i.e.,
  \begin{align*}
    &\| f - f_h \|_Q + \| (u,\qq)-(u_h,\qq_h) \| + \| (y,\xxi) - (y_h,\xxi_h) \|_U \simeq \eta.
  \end{align*}
\end{theorem}
\begin{proof}
  Let $(\wilde y,\wilde\xxi)\in U$ be given by
  \begin{align*}
    b(\wilde y,\wilde \xxi; w,\pphi) = \alpha\vdual{u_d-u_h}{w}_Q + \beta\vdual{u_{T,d}-u_h(T)}{w(T)}_\Omega  \quad\text{ for all } (w,\pphi)\in U.
  \end{align*}
  Then, 
  \begin{align}\label{eq:proof3:eq1}
    b( \wilde y-y,\wilde \xxi-\xxi; w,\pphi) =\alpha\vdual{u-u_h}{w}_Q + \beta\vdual{u(T)-u_h(T)}{w(T)}_\Omega  \quad\text{ for all } (w,\pphi)\in U,
  \end{align}
  and due to coercivity of $b$, the Cauchy--Schwarz inequality, the time-trace theorem~\eqref{eq:timetrace}, and~\eqref{eq:FK:GS}, we obtain
  \begin{align}\label{thm:apost:eq3}
    \| (y,\xxi)-(\wilde y,\wilde \xxi) \| \lesssim \| u - u_h \|_{L^2(H^1_0(\Omega))\cap H^1(H^{-1}(\Omega))}
    \lesssim \| (u,\qq) - (u_h,\qq_h) \|_U.
  \end{align}
  The triangle inequality then gives
  \begin{align*}
    \| (y,\xxi)-(y_h,\xxi_h) \|_U
    & \lesssim \| (u,\qq) - (u_h,\qq_h) \|_U + \| (\wilde y,\wilde \xxi) - (y_h,\xxi_h) \|_U.
  \end{align*}
  Next, let $(\wilde u,\wilde \qq)\in U$ be the solution of
  \begin{align*}
    b(\wilde u,\wilde\qq;v,\ppsi) = \vdual{f_h}{\partial_tv-\div\,\ppsi}_Q + \vdual{u_0}{v(0)}_\Omega \quad\text{ for all }
    (v,\ppsi)\in U.
  \end{align*}
  Then,
  \begin{align*}
    b(\wilde u - u,\wilde\qq - \qq;v,\ppsi) = \vdual{f_h-f}{\partial_tv-\div\,\ppsi}_Q \quad\text{ for all } (v,\ppsi)\in U,
  \end{align*}
  and due to coercivity of $b$ and the Cauchy--Schwarz inequality, we obtain
  \begin{align}\label{thm:apost:eq4}
    \| (u,\qq)-(\wilde u,\wilde \qq) \|_U \lesssim \| f - f_h \|_Q.
  \end{align}
  The triangle inequality then gives
  \begin{align*}
    \| (u,\qq)-(u_h,\qq_h) \|_U
    &\lesssim \| f - f_h \|_Q + \| (\wilde u,\wilde \qq)-(u_h,\qq_h) \|_U.
  \end{align*}
  Next, we will bound $\| f - f_h \|_Q$.
  To that end, define $\wilde f\in L^2(Q)$ by
  $\wilde f = \lambda^{-1}(\partial_ty_h-\div\,\xxi_h)$.
  Note that from~\eqref{eq:copt:var} it follows that $\partial_t y - \div\,\xxi=\lambda f$,
  and hence
  \begin{align}\label{thm:apost:eq1}
    \lambda (f-\wilde f) = \partial_t(y-y_h)-\div\,(\xxi-\xxi_h).
  \end{align}
  Now let $(\wat u,\wat \qq)\in U$ be the least-squares projection
  \begin{align*}
    b(\wat u,\wat\qq;v,\ppsi) = \vdual{\wilde f}{\partial_tv-\div\,\ppsi}_Q + \vdual{u_0}{v(0)}_\Omega \quad\text{ for all }
    (v,\ppsi)\in U.
  \end{align*}
  Then, 
  \begin{align}\label{eq:proof3:eq2}
    b(\wat u - u,\wat\qq - \qq;v,\ppsi) = \vdual{\wilde f-f}{\partial_tv-\div\,\ppsi}_Q \quad\text{ for all } (v,\ppsi)\in U.
  \end{align}
  Setting $(w,\pphi) = (\wat u-u,\wat\qq-\qq)$ in~\eqref{eq:proof3:eq1} and
  $(v,\ppsi) = (\wilde y-y,\wilde\xxi-\xxi)$ in~\eqref{eq:proof3:eq2} we obtain
  \begin{align}\label{thm:apost:eq2}
    \vdual{\wilde f-f}{\partial_t(\wilde y-y)-\div\,(\wilde\xxi-\xxi)}_Q
    = \alpha\vdual{u-u_h}{\wat u-u}_Q + \beta \vdual{u(T)-u_h(T)}{\wat u(T)-u(T)}_\Omega
  \end{align}
  due to the symmetry of $b$.
  Then we calculate, using~\eqref{thm:apost:eq1} and~\eqref{thm:apost:eq2},
  \begin{align*}
    \lambda \vdual{f-\wilde f}{f-\wilde f}_Q &= \vdual{f-\wilde f}{\partial_t(y-y_h)-\div\,(\xxi-\xxi_h)}_Q\\
    &= \vdual{f-\wilde f}{\partial_t(y-\wilde y)-\div\,(\xxi-\wilde \xxi)}_Q\\
    &\qquad+ \vdual{f-\wilde f}{\partial_t(\wilde y-y_h)-\div\,(\wilde\xxi-\xxi_h)}_Q\\
    &= \alpha\vdual{u-u_h}{\wat u-u}_Q + \beta \vdual{u(T)-u_h(T)}{\wat u(T)-u(T)}_\Omega\\
    &\qquad + \vdual{f-\wilde f}{\partial_t(\wilde y-y_h)-\div\,(\wilde\xxi-\xxi_h)}_Q\\
    &= \alpha\vdual{u-\wat u}{\wat u-u}_Q + \alpha\vdual{\wat u-u_h}{\wat u-u}_Q\\
    &\qquad + \beta \vdual{u(T)-\wat u(T)}{\wat u(T)-u(T)}_\Omega + \beta \vdual{\wat u(T)-u_h(T)}{\wat u(T)-u(T)}_\Omega\\
    &\qquad + \vdual{f-\wilde f}{\partial_t(\wilde y-y_h)-\div\,(\wilde\xxi-\xxi_h)}_Q\\
    &\leq \alpha \vdual{\wat u-u_h}{\wat u-u}_Q + \beta \vdual{\wat u(T)-u_h(T)}{\wat u(T)-u(T)}_\Omega\\
    &\qquad + \vdual{f-\wilde f}{\partial_t(\wilde y-y_h)-\div\,(\wilde\xxi-\xxi_h)}_Q
  \end{align*}
  where the last estimate follows since the terms $\vdual{u-\wat u}{\wat u-u}_Q$ and
  $\vdual{u(T)-\wat u(T)}{\wat u(T)-u(T)}_\Omega$ are negative. We obtain with Cauchy--Schwarz,
  the time trace theorem~\eqref{eq:timetrace}, and~\eqref{eq:FK:GS},
  \begin{align*}
  \| f - \wilde f \|_Q^2 \lesssim \|(\wat u,\wat\qq) - (u_h,\qq_h)\|_U\|(\wat u,\wat\qq) - (u,\qq) \|_U
    + \| f - \wilde f \|_Q \| (\wilde y,\wilde\xxi) - (y_h,\xxi_h) \|_U.
  \end{align*}
  Note that due to~\eqref{eq:proof3:eq2} and the coercivity of $b$, the estimate
  $\| (\wat u,\wat\qq) - (u,\qq) \|_U \lesssim \| \wilde f - f\|_Q$ holds, and analogously we have
  $\| (\wat u,\wat\qq) - (\wilde u,\wilde \qq) \|_U \lesssim \| \wilde f - f_h\|_Q$.
  We conclude
  \begin{align*}
  \| f - \wilde f \|_Q &\lesssim \|(\wat u,\wat\qq) - (u_h,\qq_h)\|_U + \| (\wilde y,\wilde\xxi) - (y_h,\xxi_h) \|_U\\
  &\lesssim \| \wilde f - f_h\|_Q + \| (\wilde u,\wilde\qq) - (u_h,\qq_h) \|_U + \| (\wilde y,\wilde\xxi) - (y_h,\xxi_h) \|_U.
  \end{align*}
  Finally,
  \begin{align*}
    \| f - f_h \|_Q &+ \| (u,\qq)-(u_h,\qq_h) \| + \| (y,\xxi) - (y_h,\xxi_h) \|_U\\
    &\lesssim \| f - f_h \|_Q + \| (\wilde u,\wilde \qq)-(u_h,\qq_h) \|_U + \| (\wilde y,\wilde\xxi) - (y_h,\xxi_h)\|_U \\
  &\lesssim \| f - \wilde f \|_Q + \|  \wilde f - f_h \|_Q + \| (\wilde u,\wilde \qq)-(u_h,\qq_h) \|_U + \| (\wilde y,\wilde\xxi) - (y_h,\xxi_h)\|_U \\
    &\lesssim \| \wilde f - f_h \|_Q + \| (\wilde u,\wilde \qq)-(u_h,\qq_h) \|_U + \| (\wilde y,\wilde\xxi) - (y_h,\xxi_h) \|_U.
  \end{align*}
  We can also prove the reverse inequality
  \begin{align*}
    &\| \wilde f - f_h \|_Q + \| (\wilde u,\wilde \qq)-(u_h,\qq_h) \|_U + \| (\wilde y,\wilde\xxi) - (y_h,\xxi_h) \|_U\\
    &\qquad\lesssim\| f - f_h \|_Q + \| (u,\qq)-(u_h,\qq_h) \| + \| (y,\xxi) - (y_h,\xxi_h) \|_U,
  \end{align*}
  which is easily seen using~\eqref{thm:apost:eq3},~\eqref{thm:apost:eq4}, and~\eqref{thm:apost:eq1}
  \begin{align*}
    \| (y_h,\xxi_h) - (\wilde y,\wilde \xxi) \|_U &\leq \| (y_h,\xxi_h) - (y,\xxi) \|_U + \| (y,\xxi) - (\wilde y,\wilde\xxi) \|_U
    \lesssim \| (y_h,\xxi_h) - (y,\xxi) \|_U + \| (u,\qq) - (u_h,\qq_h) \|_U\\
    \| (u_h,\qq_h) - (\wilde u, \wilde \qq) \|_U &\leq \| (u_h,\qq_h) - (u,\qq) \|_U + \| (u,\qq) - (\wilde u,\wilde\qq) \|_U
    \leq \| (u_h,\qq_h) - (u,\qq) \|_U + \| f-f_h \|_Q\\
    \| \wilde f - f_h \|_Q &\leq \| \wilde f - f \|_Q + \| f - f_h \|_Q \leq \| (y_h,\xxi_h) - (y,\xxi) \|_U + \| f - f_h \|_U.
  \end{align*}
  To conlude the proof, we note that $(u_h,\qq_h)$ is the least-squares projection of $(\wilde u,\wilde \qq)$, while $(y_h,\xxi_h)$ is the least-squares
  projection of $(\wilde y,\wilde\xxi)$. Hence,
  \begin{align*}
    \| (\wilde u,\wilde \qq)-(u_h,\qq_h) \|_U &\simeq \| \qq_h-\nabla u_h \|_{Q} + \| \partial_t u_h - \div\,\qq_h-f_h \|_Q + \| u_h(0) - u_0 \|_\Omega,\\
    \| (\wilde y,\wilde \xxi)-(y_h,\xxi_h) \|_U &\simeq \| \xxi_h-\nabla y_h - \wilde p  \|_{Q} + \| \partial_t y_h - \div\,\xxi_h+\wilde p \|_Q,
  \end{align*}
  concludes the proof.
\end{proof}
To obtain a fully computable estimator we have to approximate $\wilde p$ in Theorem~\ref{thm:apost}: Following the spirit of this paper we
consider a least-squares discretization. Changing signs accordingly, we may apply the same arguments as in Section~\ref{sec:heat}: Let 
\begin{align*}
  U^\star :=
  \left\{ (v,\ppsi)\mid v \in L^2(H^1_0(\Omega)), \ppsi\in L^2(Q), -\partial_t v -\div\,\ppsi\in L^2(Q) \right\},
\end{align*}
and note that $U^\star$ is a Hilbert space with norm 
\begin{align*}
  \|(v,\ppsi)\|_{U^\star}^2 = \|v\|_{L^2(H^1(\Omega))}^2 + \|\partial_t v+\div\,\ppsi\|_Q^2 + \|v(T)\|_\Omega^2,
\end{align*}
cf.~\cite[Sec.~2]{GantnerS_21}.
The first-order system for the continuous problem from Theorem~\ref{thm:apost} then is
\begin{subequations}\label{eq:adjointTilde}
\begin{alignat}{2}
  -\partial_t \wilde p - \div\,\wilde{\boldsymbol\chi} &= \alpha(u_h-u_d) &\quad&\text{in }Q, \\
  \nabla \wilde p - \wilde{\boldsymbol\chi} &= 0 &\quad&\text{in }Q, \\ 
  \wilde p&= 0 &\quad&\text{on }J\times \partial\Omega, \\
  \wilde p(T) &= \beta(u_h(T)-u_{T,d}) &\quad&\text{in }\Omega.
\end{alignat}
\end{subequations}
This problem admits a unique solution $(\wilde p,\wilde{\boldsymbol\chi})\in U^\star$. Moreover, the left-hand side of~\eqref{eq:adjointTilde} defines a boundedly invertible operator, i.e.,
\begin{align*}
  \|(v,\ppsi)\|_{U^\star}^2 \simeq \|\nabla v-\ppsi\|_Q^2 + \|\partial_t v+\div\,\ppsi\|_Q^2 + \|v(T)\|_\Omega^2.
\end{align*}
In particular, the solution of~\eqref{eq:adjointTilde} satisfies
\begin{align*}
  \|(\wilde p,\wilde{\boldsymbol\chi}\|_{U^\star} \lesssim \|u_h-u_d\|_Q + \|u_h(T)-u_{T,d}\|_\Omega.
\end{align*}
As already mentioned, the proof of the latter claims follow the argumentations as presented in Section~\ref{sec:heat} and are therefore omitted.
\begin{cor}
  Under the assumptions and notations from Theorem~\ref{thm:apost}, let $(p,\boldsymbol\chi):=(p,\nabla p)\in U^\star$ where $p$
  denotes the solution of the adjoint problem~\eqref{eq:copt:transpose} and let $(p_h,{\boldsymbol\chi}_h)\in U^\star$ be arbitrary. 
  Define 
  \begin{align*}
    \widetilde\eta^2 &= \| \lambda^{-1}(\partial_ty_h-\div\,\xxi_h) - f_h \|_Q^2 + \| \qq_h-\nabla u_h \|_{Q}^2
    + \| \partial_t u_h - \div\,\qq_h-f_h \|_Q^2 + \| u_h(0) - u_0 \|_\Omega^2\\
    &\qquad+ \| \xxi_h-\nabla y_h - {\boldsymbol\chi}_h \|_{Q}^2 + \| \partial_t y_h - \div\,\xxi_h+ p_h \|_Q^2 \\
    &\qquad+ \|\alpha(u_h-u_d) + \partial_t  p_h + \div{\boldsymbol\chi}_h\|_Q^2
    + \|\nabla  p_h-{\boldsymbol\chi}_h\|_Q^2 + \| p_h(T) -\beta(u_h(T)-u_d)\|_\Omega^2.
  \end{align*}
  Then, $\wilde\eta$ is reliable and efficient, i.e.,
  \begin{align*}
    \| f - f_h \|_Q + \| (u,\qq)-(u_h,\qq_h) \|_U + \| (y,\xxi) - (y_h,\xxi_h) \|_U +
    \|(p,{\boldsymbol\chi})-(p_h,{\boldsymbol\chi}_h)\|_{U^\star} \simeq \widetilde\eta.
  \end{align*}
\end{cor}
\begin{proof}
  Note the equivalence
  \begin{align*}
    &\|(\widetilde p,\widetilde{\boldsymbol\chi})-(p_h,{\boldsymbol\chi}_h)\|_{U^\star}^2 
    \\&\qquad \simeq
    \|\alpha(u_h-u_d) + \partial_t  p_h + \div\,{\boldsymbol\chi}_h\|_Q^2
    + \|\nabla p_h-{\boldsymbol\chi}_h\|_Q^2 + \| p_h(T) -\beta(u_h(T)-u_d)\|_\Omega^2
  \end{align*}
  for any $(p_h,{\boldsymbol\chi}_h)\in U^\star$, where $(\widetilde p,\widetilde{\boldsymbol\chi})\in U^\star$ is the solution of~\eqref{eq:adjointTilde}.
  By the triangle inequality we have that
    \begin{align}\label{eq:equivEtaTildeEta}
      \widetilde\eta^2 \simeq \eta^2 + \|(\widetilde p,\widetilde{\boldsymbol\chi})-(p_h,\boldsymbol\chi_h)\|_{U^\star}^2.
    \end{align}
    Using the triangle inequality and continuous dependence on the data we infer that
    \begin{align*}
      \|(p,\boldsymbol\chi)-(p_h,\boldsymbol\chi_h)\|_{U^\star} &\leq \|(p,{\boldsymbol\chi})-(\widetilde p,\widetilde{\boldsymbol\chi})\|_{U^\star} + \|(\widetilde p,\widetilde{\boldsymbol\chi})-(p_h,\boldsymbol\chi_h)\|_{U^\star}
      \\
      &\lesssim \|u-u_h\|_Q + \|u(T)-u_h(T)\|_\Omega + \|(\widetilde p,\widetilde{\boldsymbol\chi})-(p_h,\boldsymbol\chi_h)\|_{U^\star}
      \\
      &\lesssim \|(u,\qq)-(u,\qq_h)\|_U + \|(\widetilde p,\widetilde{\boldsymbol\chi})-(p_h,\boldsymbol\chi_h)\|_{U^\star}.
    \end{align*}
    Similarily, we obtain that
    \begin{align*}
      \|(\widetilde p,\widetilde{\boldsymbol\chi})-(p_h,\boldsymbol\chi_h)\|_{U^\star} \lesssim \|(u,\qq)-(u,\qq_h)\|_U + \|(p,\boldsymbol\chi)-(p_h,\boldsymbol\chi_h)\|_{U^\star}.
    \end{align*}
    Combining the latter two estimates together with~\eqref{eq:equivEtaTildeEta} and reliability as well as efficiency of $\eta$ (Theorem~\ref{thm:apost}) we conclude the proof.
\end{proof}
Clearly, in the last result we want that $(p_h,{\boldsymbol\chi}_h)\in U^\star$ is an approximation of the exact solution of~\eqref{eq:adjointTilde}.
\subsection{A lowest order discretization}
Let $\tT_h$ be a simplicial and admissible partition of the space-time cylinder $J\times\Omega$.
By admissible, we mean that there are no hanging nodes. Define
\begin{align*}
  \pP^0(\tT_h) &= \left\{ u\in L^2(J\times\Omega) \mid \forall K\in \tT_h:\; u|_K \text{ is constant} \right\},\\
  \sS^1(\tT_h) &= \left\{ u \in C(J\times\Omega) \mid u|_K \text{ a polynomial of degree at most } 1 \text{ for all
  } K\in \tT_h \right\},\\
  \sS^1_0(\tT_h) &= \left\{ u \in \sS^1(\tT_h) \mid u = 0 \text{ on } \overline J\times\partial\Omega  \right\}.
\end{align*}
Note that $\sS^1_0(\tT_h)$ is a subspace of $L^2(J;H^1_0(\Omega))\cap H^1(J;H^{-1}(\Omega))$,
and $[\sS^1(\tT_h)]^d$ is a subspace of $L^2(J\times\Omega)^d$.
Furthermore, if $v_h\in\sS^1_0(\tT_h)$ and $\ppsi_h\in[\sS^1(\tT_h)]^d$, then
$\partial_tv_h - \div\,\ppsi_h\in L^2(J\times \Omega)$.
Hence, we can define the discrete conforming subspaces
\begin{align}\label{eq:discretespaces}
  \begin{split}
  U_h &= \sS^1_0(\tT_h)\times [\sS^1(\tT_h)]^d \subset U,\\
  V_h &= \pP^0(\tT_h)\times U_h\subset V.
  \end{split}
\end{align}
In order to obtain a-priori error estimates, we will have to use structured meshes.
To that end, let $J_h=\{ j_1,j_2,\dots \}$ with $j_k=(t_k,t_{k+1})$ be a partition of the time interval $J$ into subintervals of length $h$,
and let $\Omega_h = \{\omega_1,\omega_2,\dots\}$ be a partition of physical space $\Omega\subset\R^d$ into $d$-simplices of diameter $h$.
The tensor-product mesh $J_h\otimes\Omega_h$ consists of elements which are space-time cylinders with
$d$-simplices from $\Omega_h$ as base. It is possible to construct from $J_h\otimes\Omega_h$ a simplicial, admissible mesh $\tT_h$,
following the recent work~\cite{NK_15}. To that end, suppose that the vertices of $\Omega_h$ are numbered like
$p_1,p_2, \dots, p_N$. An element $\omega\in\Omega_h$ can then be represented uniquely as the convex hull
\begin{align*}
  \omega = \textrm{conv}(p_{i^\omega_1}, \dots, p_{i^\omega_{d+1}}), \text{ with } i^\omega_k < i^\omega_\ell \text{ for } k<\ell.
\end{align*}
This ``local numbering'' of vertices is called \textit{consistent numbering} in the literature, cf.~\cite{Bey_00}.
An element $K=\omega\times j_k \in J_h\otimes\Omega_h$ can hence be written as convex hull
\begin{align*}
  K = \textrm{conv}(p_1', \dots, p_{d+1}',p_1'',\dots,p_{d+1}''),
\end{align*}
where
\begin{align*}
  p_\ell' = (p_{i^\omega_\ell},t_k), \quad p_\ell'' = (p_{i^\omega_\ell},t_{k+1}).
\end{align*}
We can split $K$ into $(d+1)$ different $d+1$-simplices
\begin{align}
  \begin{split}\label{int:simp:split}
  \textrm{conv}&(p_1',p_2',\dots, p_{d+1}',p_1'')\\
  \textrm{conv}&(p_2',\dots, p_{d+1}',p_1'',p_2'')\\
  &\vdots\\
  \textrm{conv}&(p_{d+1}',p_1'',p_2'',\dots,p_{d+1}''),
  \end{split}
\end{align}
and this way we obtain a simplicial triangulation $\tT_h$ of $J\times \Omega$.
The following result from~\cite[Thm.~1]{NK_15} holds.
\begin{theorem}
  The simplicial partition $\tT_h$ is admissible.
\end{theorem}
In~\cite[Thm.~11, Thm.~12]{FuehrerK_21}, we prove the following approximation estimates.
\begin{theorem}\label{thm:heat:interpolation}
  Suppose that $\tT_h$ is a simplicial mesh constructed from a tensor product mesh $J_h\otimes \Omega_h$.
  Then, there exist linear operators
  \begin{align*}
    \Jj_h:L^2(H^1(\Omega))\cap H^1(H^{-1}_0(\Omega))\rightarrow \sS^1(\tT_h),\\
    \Jj_{0,h} :L^2(H^1_0(\Omega))\cap H^1(H^{-1}(\Omega))\rightarrow \sS^1_0(\tT_h),
  \end{align*}
  such that
  \begin{align*}
    \| v-\Jj_{0,h}v \|_{L^2(H^1_0(\Omega))} + \| (v-\Jj_{0,h}v)' \|_{L^2(L^2(\Omega))}
    &\lesssim h \left( \| v \|_{H^1(H^1_0(\Omega))} + \| v \|_{H^2(L^2(\Omega))} + \| v \|_{L^\infty(H^2(\Omega))} \right).\\
    \| \nabla v - \Jj_h\nabla v \|_{L^2(H^1(\Omega))}
    &\lesssim h \left( \| v \|_{H^1(H^2(\Omega))} + \| v \|_{L^\infty(H^3(\Omega))} \right).
  \end{align*}
\end{theorem}
For recent and improved results on tensor meshes we refer to~\cite[Sec.~4]{DieningST_21}.
\begin{theorem}\label{thm:copt:apriori}
  Let $(f,u,\qq)\in V$ and $(y,\xxi)\in U$ be the solution of Problem~\eqref{eq:copt:var}, and $p$ be the adjoint
  state. Suppose that $\tT_h$ is a simplicial mesh constructed from a tensor product mesh $J_h\otimes \Omega_h$.
  For the discrete spaces
  $U_h,V_h$ defined in~\eqref{eq:discretespaces}, let $(f_h,u_h,\qq_h)\in V_h$ and $(y_h,\xxi_h)\in U_h$ be the solution
  of Problem~\eqref{eq:copt:disc}.
  Suppose that $f\in H^1(Q)$ and
  \begin{align*}
    u,p &\in L^2(H^4(\Omega))\cap H^1(H^2(\Omega))\cap H^2(L^2(\Omega)).
  \end{align*}
  Then
  \begin{align*}
    \| (f,u,\qq) - (f_h,u_h,\qq_h) \|_V + \| (y,\xxi) - (y_h,\xxi_h) \|_U
    = \OO(h).
  \end{align*}
\end{theorem}
\begin{proof}
  According to Corollary~\ref{cor:schwarz} and the embedding~\eqref{eq:timetrace:2} with $m=2$, we see
  \begin{align*}
    u,p\in H^1(H^1_0(\Omega)\cap H^2(\Omega)) \cap H^2(L^2(\Omega)) \cap L^\infty(H^3(\Omega)).
  \end{align*}
  According to Lemma~\ref{lem:lagmult}, $\qq=\nabla u$ and $\xxi = \nabla(y+p)$.
  Denote by $\Pi^0_h:L^2(Q)\rightarrow \pP^0(\tT_h)$ the $L^2(Q)$-orthogonal projection.
  Theorem~\ref{thm:heat:interpolation} and standard approximation estimates for $\Pi^0_h$ show
  \begin{align*}
    \| (f,u,\qq)-(\Pi_hf,\Jj_{0,h}u,\Jj_h\nabla u) \|_V
    &\lesssim \| f - \Pi_h f\|_{L^2(Q)} + \| (u,\nabla u)-(\Jj_{0,h}u,\Jj_h\nabla u) \|_U\\
    &\lesssim \| f - \Pi_h f\|_{L^2(Q)} + \| u-\Jj_{0,h}u \|_{L^2(H^1_0(\Omega))} + \| (u-\Jj_{0,h}u)' \|_{L^2(L^2(\Omega))}\\
    &\qquad + \| \nabla u - \Jj_h\nabla u \|_{L^2(L^2(\Omega))} + \| \div(\nabla u - \Jj_h\nabla u) \|_{L^2(L^2(\Omega))}\\
    &\lesssim h \left( \| f \|_{H^1(Q)} +
    \| u \|_{H^1(H^1_0(\Omega)\cap H^2(\Omega))} + \| u \|_{H^2(L^2(\Omega))} + \| u \|_{L^\infty(H^3(\Omega))} \right).
  \end{align*}
  Next, Lemma~\ref{lem:copt:regularity} and Corollary~\ref{cor:schwarz} imply also that
  \begin{align*}
    y\in H^1(H^1_0(\Omega)\cap H^2(\Omega)) \cap H^2(L^2(\Omega)) \cap L^\infty(H^3(\Omega)),
  \end{align*}
  and hence
  \begin{align*}
    \| (y,\xxi) - (\Jj_{0,h}y,\Jj_h\nabla(y+p)) \|_U &\lesssim \| (y,\nabla y) - (\Jj_{0,h}y,\Jj_h\nabla y) \|_U\\
    &\qquad+ \| \nabla p - \Jj_h\nabla p \|_{L^2(L^2(\Omega))} + \| \div(\nabla p - \Jj_h\nabla p) \|_{L^2(L^2(\Omega))}\\
    &\lesssim
    h\Bigl( 
      \| y \|_{H^1(H^1_0(\Omega)\cap H^2(\Omega))} + \| y \|_{H^2(L^2(\Omega))} + \| y \|_{L^\infty(H^3(\Omega))}\\
      &\qquad\qquad + \| p \|_{H^1(H^2(\Omega))} + \| p \|_{L^\infty(H^3(\Omega))}
      \Bigr)
  \end{align*}
\end{proof}
Combining Lemma~\ref{lem:copt:regularity} and Theorem~\ref{thm:copt:apriori}, we immediately obtain
\begin{cor}\label{cor:copt:apriori}
  Let $\Omega\subset\R^d$ with $\partial\Omega$ be smooth, or, particularly, $\Omega\subset\R$ an interval.
  Suppose that $u_0\in H^1_0(\Omega)\cap H^3(\Omega)$ with $\Delta u_0\in H^1_0(\Omega)$, and
  $u_d\in L^2(H^2(\Omega))\cap H^1(L^2(\Omega))$ with $u_d(T)\in H^1_0(\Omega)$,
  and $u_{T,d}\in H^1_0(\Omega)$ with $\Delta u_{T,d}\in H^1_0(\Omega)$.
  Let $(f,u,\qq)\in V$ and $(y,\xxi)\in U$ be the solution of Problem~\eqref{eq:copt:var}
  For the discrete spaces $U_h,V_h$ defined in~\eqref{eq:discretespaces}, let $(f_h,u_h,\qq_h)\in V_h$ and $(y_h,\xxi_h)\in U_h$ be the solution
  of Problem~\eqref{eq:copt:disc}.
  Then
  \begin{align*}
    \| (f,u,\qq) - (f_h,u_h,\qq_h) \|_V + \| (y,\xxi) - (y_h,\xxi_h) \|_U
    = \OO(h).
  \end{align*}
\end{cor}
\section{Numerical Experiments}
In this section we present several numerical experiments for the case $d=1$.
In all experiments, we use $\Omega = (0,1)$ and $J = (0,T)$ and $\lambda=1$.
For the discrete spaces $V_h,U_h$ in~\eqref{eq:discretespaces} and
$(f_h,u_h,\qq_h,y_h,\xxi_h) \in V_h \times U_h$ being the discrete aproximation~\eqref{eq:copt:disc}
of~\eqref{eq:copt:var},
we define the error estimator
\begin{align*}
  \wilde\eta_h^2 := \wilde\eta^2(f_h,u_h,\qq_h,y_h,\xxi_h)
  = \sum_{K\in\tT_h} \wilde\eta_K^2(f_h,u_h,\qq_h,y_h,\xxi_h)^2
\end{align*}
where the local error indicators are given by
\begin{align*}
  &\wilde\eta_K(f_h,u_h,\qq_h,y_h,\xxi_h)^2 := \\
  &\qquad\| \lambda^{-1}(\partial_ty_h-\div\,\xxi_h) - f_h \|_Q^2 + \| \qq_h-\nabla u_h \|_{Q}^2
  + \| \partial_t u_h - \div\,\qq_h-f_h \|_Q^2 + \| u_h(0) - u_0 \|_{\partial K\cap\left\{ 0 \right\}\times\Omega }^2\\
    &\qquad+ \| \xxi_h-\nabla y_h - {\boldsymbol\chi}_h \|_{Q}^2 + \| \partial_t y_h - \div\,\xxi_h+ p_h \|_Q^2 \\
    &\qquad+ \|\alpha(u_h-u_d) + \partial_t  p_h + \div{\boldsymbol\chi}_h\|_Q^2
    + \|\nabla  p_h-{\boldsymbol\chi}_h\|_Q^2 + \| p_h(T) -\beta(u_h(T)-u_d)\|_{\partial K\cap\left\{ T \right\}\times\Omega }^2,
\end{align*}
with $(p_h,\xxi_h)\in U_h$ being the least-squares approximation of~\eqref{eq:adjointTilde}. In our implementation we
replace $u_d$ by $\Pi_h^0 u_d$, where $\Pi_h^0:L^2(Q)\rightarrow \pP^0(\tT_h)$ is the $L^2(Q)$-orthogonal projection, as well
as $u_0$ by $\Pi_{0,h} u_0$ and $u_{T,d}$ by $\Pi_{T,h}^1 u_{T,d}$, where $\Pi_{t,h}^1:L^2(\Omega) \rightarrow \sS^1(\tT_h\cap t \times \Omega)$ is
the $L^2(\Omega)$-orthogonal projection on the spatial mesh given by the restriction of $\tT_h$ to time $t$. Therefore, we also include data oscillation terms
\begin{align*}
  \osc_h^2 := \sum_{K\in\tT_h} \osc_K^2,
\end{align*}
where the local oscillation terms are given by
\begin{align*}
  \osc_K^2 := \| u_d - \Pi_h^0 u_d \|_{K}^2 + \| u_{T,d} - \Pi_{T,h}^1 u_{T,d} \|_{\partial K\cap\left\{ T \right\}\times\Omega }^2
  + \| u_0 - \Pi_{0,h}^1 u_0 \|_{\partial K\cap\left\{ 0 \right\}\times\Omega }^2.
\end{align*}
Finally, we use
\begin{align*}
  \ind_h^2 := \wilde\eta_h^2 + \osc_h^2 = \sum_{K\in\tT_h} \wilde\eta_K^2 + \osc_K^2
\end{align*}
as refinement indicators to steer the adaptive mesh-refinement. We use the D\"orfler criterion to mark elements for refinement, i.e.,
find a minimal set $\mM\subset\tT_h$ such that
\begin{align*}
  \theta\; \ind_h^2 \leq \sum_{K\in\tT_h} \ind_K^2.
\end{align*}
Throughout, we use $\theta=0.5$. Starting from an initial mesh, we use Newest Vertex Bisection to generate uniform and adaptively refined meshes.
In all figures we visualize convergence rates using triangles where the (negative) slope is indicated
by a number besides the triangle. We plot quantities of interest with respect to the overall number of degrees of freedom, which is (counting also
Dirichlet degrees of freedom)
\begin{align*}
  \text{degrees of freedom} = (\text{number of elements}) + 4(\text{number of nodes}).
\end{align*}
Note that for uniform refinement we have
\begin{align*}
  \text{degrees of freedom}^{-1} \simeq h^2.
\end{align*}
All computations have been realized on an Intel i7-10750H 2.60GHz Linux machine with 16GB RAM,
implemented in MATLAB using sparse matrices and the MATLAB backslash operator to solve linear systems.
\subsection{Experiment 1}\label{ex1}
We use the smooth exact solution
\begin{align*}
  u(x,t) &= \sin(\pi x)\cos(\pi t),\\
  f(x,t) &= -\pi \sin(\pi x) \sin(\pi t) + \pi^2\sin(\pi x)\cos(\pi t),
\end{align*}
the parameters $\alpha=\beta=1/2$,
and calculate the data $u_d$ and $u_{T,d}$ using Theorem~\ref{thm:copt}.
The convergence histories for $\| u - u_h \|_{L^2(Q)}^2$ and $\| f - f_h \|_{L^2(Q)}^2$ in the
case of uniform and adaptive mesh refinement are shown in Figures~\ref{fig:Example1Unif}
and~\ref{fig:Example1Adap}, together with the error estimator $\widetilde\eta_h$ and the data
oscillation term $\osc_h$.
According to Corollary~\ref{cor:copt:apriori}, we expect in the uniform case the optimal rate
\begin{align*}
  \| u - u_h \|_{L^2(Q)}^2 + \| f - f_h \|_{L^2(Q)}^2 \leq
  \| (f,u,\qq) - (f_h,u_h,\qq_h) \|_V^2 + \| (y,\xxi) - (y_h,\xxi_h) \|_U^2 = \OO(h^2),
\end{align*}
which we observe in Figure~\ref{fig:Example1Unif}. We even observe the optimal convergence rate for
the $L^2$-error of $u$, $\| u - u_h \|_{L^2(Q)}^2 = \OO(h^4)$. The convergence history in the
case of adaptive mesh refinement in Figure~\ref{fig:Example1Adap} shows that
the adaptive algorithm reproduces optimal rates. In Figure~\ref{fig:Example1AdapMesh}
we plot intermediate adaptive meshes with $160$ respectively $4321$ number of elements.
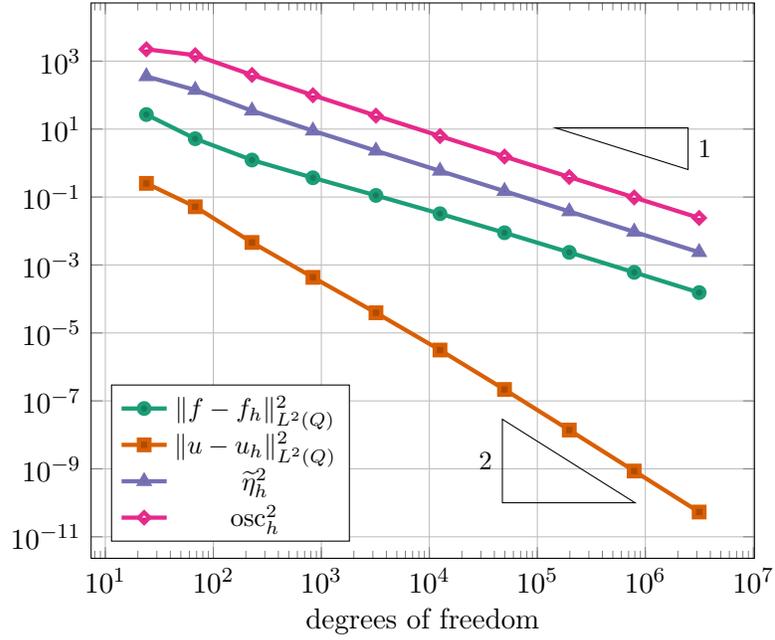
\begin{figure}
  \begin{center}
      \begin{tikzpicture}
\begin{loglogaxis}[
width=0.65\textwidth,
cycle list/Dark2-6,
cycle multiindex* list={
mark list*\nextlist
Dark2-6\nextlist},
every axis plot/.append style={ultra thick},
xlabel={degrees of freedom},
grid=major,
legend entries={\small $\| f - f_h \|_{L^2(Q)}^2$,\small $\| u - u_h \|_{L^2(Q)}^2$,
\small $\widetilde \eta_h^2$, \small $\osc_h^2$},
legend pos=south west,
]
\addplot table [x=ndofs,y=errsqfL2] {figures/Example1Unif.dat};
\addplot table [x=ndofs,y=errsquL2] {figures/Example1Unif.dat};
\addplot table [x=ndofs,y=est] {figures/Example1Unif.dat};
\addplot table [x=ndofs,y=osc] {figures/Example1Unif.dat};

\logLogSlopeTriangle{0.9}{0.2}{0.7}{1}{black}{{\small $1$}};
\logLogSlopeTriangleBelow{0.82}{0.2}{0.1}{2}{black}{{\small $2$}};
\end{loglogaxis}
\end{tikzpicture}
  \end{center}
  \caption{Uniform mesh-refinement for Experiment 1.}
  \label{fig:Example1Unif}
\end{figure}
\begin{figure}
  \begin{center}
      \begin{tikzpicture}
\begin{loglogaxis}[
width=0.65\textwidth,
cycle list/Dark2-6,
cycle multiindex* list={
mark list*\nextlist
Dark2-6\nextlist},
every axis plot/.append style={ultra thick},
xlabel={degrees of freedom},
grid=major,
legend entries={\small $\| f - f_h \|_{L^2(Q)}^2$,\small $\| u - u_h \|_{L^2(Q)}^2$,
\small $\widetilde \eta_h^2$, \small $\osc_h^2$},
legend pos=south west,
]
\addplot table [x=ndofs,y=errsqfL2] {figures/Example1Adap.dat};
\addplot table [x=ndofs,y=errsquL2] {figures/Example1Adap.dat};
\addplot table [x=ndofs,y=est] {figures/Example1Adap.dat};
\addplot table [x=ndofs,y=osc] {figures/Example1Adap.dat};

\logLogSlopeTriangle{0.9}{0.2}{0.7}{1}{black}{{\small $1$}};
\logLogSlopeTriangleBelow{0.82}{0.2}{0.1}{2}{black}{{\small $2$}};
\end{loglogaxis}
\end{tikzpicture}
  \end{center}
  \caption{Adaptive mesh-refinement for Experiment 1.}
  \label{fig:Example1Adap}
\end{figure}
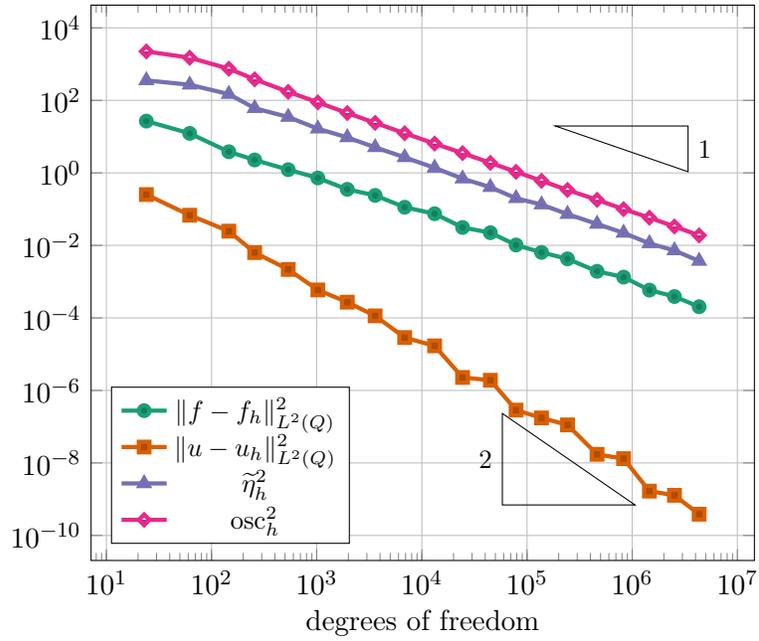
\begin{figure}
  \begin{center}
      \begin{tikzpicture}
\begin{axis}[
width=0.49\textwidth,
    axis equal,
]

\addplot[patch,color=white,
faceted color = black, line width = 0.5pt,
patch table ={figures/Example1_el1.dat}] file{figures/Example1_co1.dat};
\end{axis}
\end{tikzpicture}
      \begin{tikzpicture}
\begin{axis}[
width=0.49\textwidth,
    axis equal,
]

\addplot[patch,color=white,
faceted color = black, line width = 0.5pt,
patch table ={figures/Example1_el2.dat}] file{figures/Example1_co2.dat};
\end{axis}
\end{tikzpicture}
  \end{center}
  \caption{Adaptive meshes for Experiment 1 with $160$ resp. $4321$ elements.}
  \label{fig:Example1AdapMesh}
\end{figure}
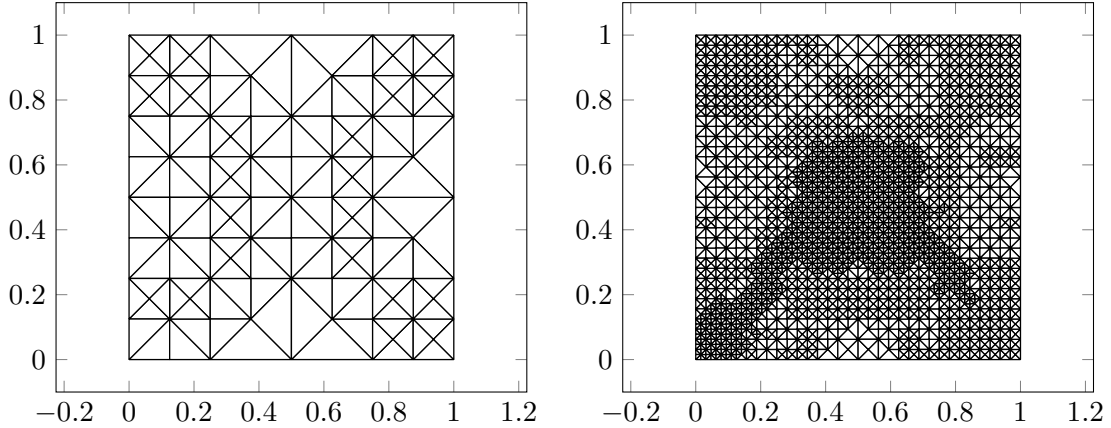
\subsection{Experiment 2}
We use $u_0=0$, $\alpha=1,\beta=0$ and the discontinuous desired state
\begin{align*}
  u_d(t,x) = 
  \begin{cases}
    1 & 0.5\leq t \text{ and } 0.2\leq x \leq 0.8,\\
    0 & \text{ else}.
  \end{cases}
\end{align*}
Note that $u_d$ does not fulfill the conditions of Corollary~\ref{cor:copt:apriori}, and consequently we observe
a reduced convergence rate for $\wilde\eta_h$ and $\osc_h$ in the case of uniform mesh-refinement, cf. Figure~\ref{fig:Example2},
from which we conclude
\begin{align*}
  \| (f,u,\qq) - (f_h,u_h,\qq_h) \|_V^2 + \| (y,\xxi) - (y_h,\xxi_h) \|_U^2 \simeq \eta_h^2 \simeq \OO(h).
\end{align*}
The adaptive algorithm, on the other hand, recovers the optimal convergence rate, cf. Figure~\ref{fig:Example2}.
In Figure~\ref{fig:Example2AdapMesh} we plot two intermediate adaptive meshes with $379$ respectively $4730$ elements, 
and observe that the adaptive algorithm refines the mesh in the vicinity of the discontinuity of the desired state $u_d$.
We also observe that the spatial discontinuity seems to require a finer resolution than the temporal discontinuity.

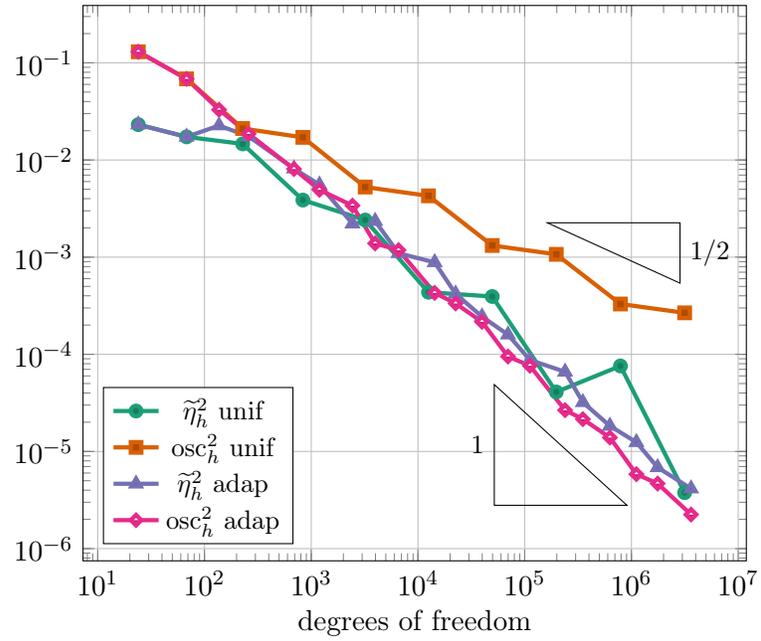
\begin{figure}
  \begin{center}
      \begin{tikzpicture}
\begin{loglogaxis}[
width=0.65\textwidth,
cycle list/Dark2-6,
cycle multiindex* list={
mark list*\nextlist
Dark2-6\nextlist},
every axis plot/.append style={ultra thick},
xlabel={degrees of freedom},
grid=major,
legend entries={\small $\widetilde \eta_h^2$ unif, \small $\osc_h^2$ unif,
\small $\widetilde \eta_h^2$ adap, \small $\osc_h^2$ adap},
legend pos=south west,
]
\addplot table [x=ndofs,y=est] {figures/Example2Unif.dat};
\addplot table [x=ndofs,y=osc] {figures/Example2Unif.dat};
\addplot table [x=ndofs,y=est] {figures/Example2Adap.dat};
\addplot table [x=ndofs,y=osc] {figures/Example2Adap.dat};

\logLogSlopeTriangle{0.9}{0.2}{0.5}{0.5}{black}{{\small $1/2$}};
\logLogSlopeTriangleBelow{0.82}{0.2}{0.1}{1}{black}{{\small $1$}};
\end{loglogaxis}
\end{tikzpicture}
  \end{center}
  \caption{Uniform and adaptive mesh-refinement for Experiment 2.}
  \label{fig:Example2}
\end{figure}
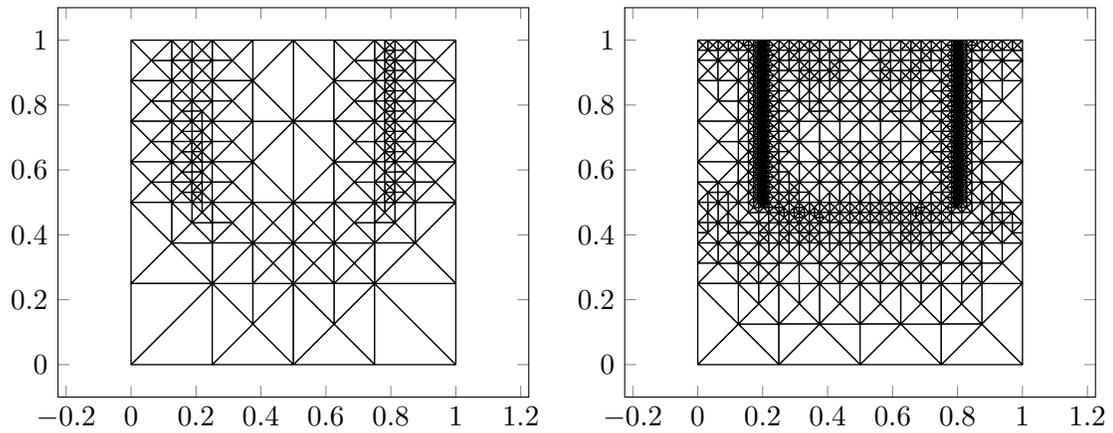
\begin{figure}
  \begin{center}
      \begin{tikzpicture}
\begin{axis}[
width=0.49\textwidth,
    axis equal,
]

\addplot[patch,color=white,
faceted color = black, line width = 0.5pt,
patch table ={figures/Example2_el1.dat}] file{figures/Example2_co1.dat};
\end{axis}
\end{tikzpicture}
      \begin{tikzpicture}
\begin{axis}[
width=0.49\textwidth,
    axis equal,
]

\addplot[patch,color=white,
faceted color = black, line width = 0.5pt,
patch table ={figures/Example2_el2.dat}] file{figures/Example2_co2.dat};
\end{axis}
\end{tikzpicture}
  \end{center}
  \caption{Adaptive mehses for Experiment 2 with $379$ resp. $4730$ elements.}
  \label{fig:Example2AdapMesh}
\end{figure}

\subsection{Experiment 3}
We use $u_0=0$, $\alpha=0, \beta=1$ and the smooth desired state
\begin{align*}
  u_{T,d}(x) = \sin(\pi x).
\end{align*}
According to Corollary~\ref{cor:copt:apriori}, we expect in the uniform case the optimal rates
\begin{align*}
  \eta_h^2 \simeq \| (f,u,\qq) - (f_h,u_h,\qq_h) \|_V^2 + \| (y,\xxi) - (y_h,\xxi_h) \|_U^2 &= \OO(h^2),\\
  \osc_h^2 = \| u_{T,d} - \Pi_{T,h}^1 u_{T,d} \|_{L^2(\Omega)}^2 &\simeq \OO(h^4),
\end{align*}
which we observe in Figure~\ref{fig:Example3}.
The convergence history in the case of adaptive mesh refinement in Figure~\ref{fig:Example3} shows that
the adaptive algorithm reproduces the optimal rates. In Figure~\ref{fig:Example3AdapMesh}
we plot intermediate adaptive meshes with $142$ respectively $3366$ number of elements.

\begin{figure}
  \begin{center}
      \begin{tikzpicture}
\begin{loglogaxis}[
width=0.65\textwidth,
cycle list/Dark2-6,
cycle multiindex* list={
mark list*\nextlist
Dark2-6\nextlist},
every axis plot/.append style={ultra thick},
xlabel={degrees of freedom},
grid=major,
legend entries={\small $\widetilde \eta_h^2$ unif, \small $\osc_h^2$ unif,
\small $\widetilde \eta_h^2$ adap, \small $\osc_h^2$ adap},
legend pos=south west,
]
\addplot table [x=ndofs,y=est] {figures/Example3Unif.dat};
\addplot table [x=ndofs,y=osc] {figures/Example3Unif.dat};
\addplot table [x=ndofs,y=est] {figures/Example3Adap.dat};
\addplot table [x=ndofs,y=osc] {figures/Example3Adap.dat};

\logLogSlopeTriangle{0.9}{0.2}{0.7}{1}{black}{{\small $1$}};
\logLogSlopeTriangleBelow{0.82}{0.2}{0.1}{2}{black}{{\small $2$}};
\end{loglogaxis}
\end{tikzpicture}
  \end{center}
  \caption{Uniform and adaptive mesh-refinement for Experiment 3.}
  \label{fig:Example3}
\end{figure}
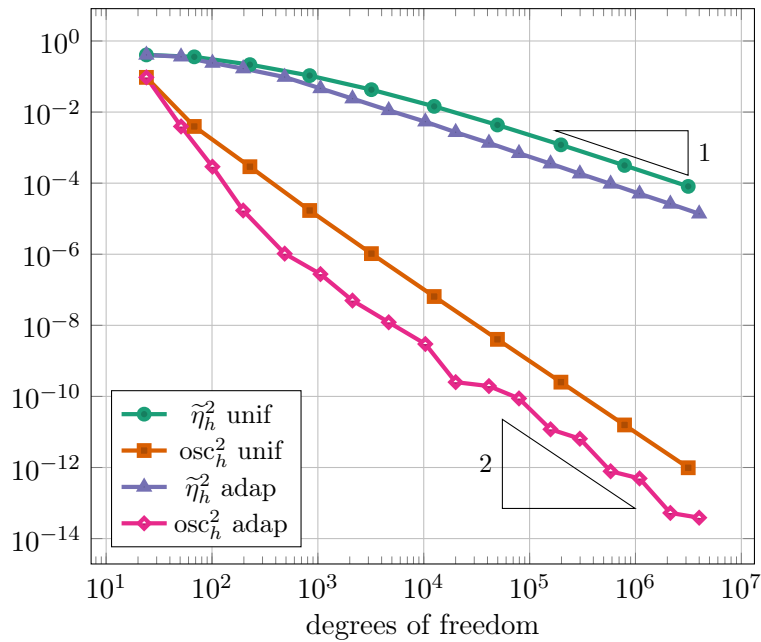
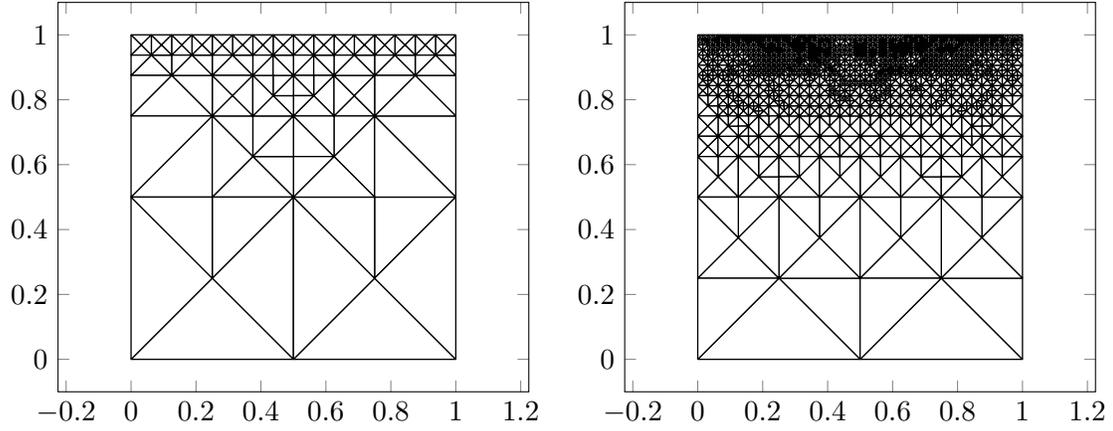
\begin{figure}
  \begin{center}
      \begin{tikzpicture}
\begin{axis}[
width=0.49\textwidth,
    axis equal,
]

\addplot[patch,color=white,
faceted color = black, line width = 0.5pt,
patch table ={figures/Example3_el1.dat}] file{figures/Example3_co1.dat};
\end{axis}
\end{tikzpicture}
      \begin{tikzpicture}
\begin{axis}[
width=0.49\textwidth,
    axis equal,
]

\addplot[patch,color=white,
faceted color = black, line width = 0.5pt,
patch table ={figures/Example3_el2.dat}] file{figures/Example3_co2.dat};
\end{axis}
\end{tikzpicture}
  \end{center}
  \caption{Adaptive meshes for Experiment 3 with $142$ resp $3366$ elements.}
  \label{fig:Example3AdapMesh}
\end{figure}
\subsection{Experiment 4}
We use $u_0=0$, $\alpha=0, \beta=1$ and the desired state
\begin{align*}
  u_{T,d}(x) = 
  \begin{cases}
    0.1 - |x-0.5| & 0.4\leq x \leq 0.6,\\
    0 & \text{ else}.
  \end{cases}
\end{align*}
Note that $u_{T,d}\in H^1_0(\Omega)$ but $\Delta u_{T,d}\notin H^1(\Omega)$. Consequently we observe
a reduced convergence rate for $\wilde\eta_h$ and in the case of uniform mesh-refinement, cf. Figure~\ref{fig:Example4},
while $\osc_h$ converges with optimal rate. For adaptive mesh-refinement, Figure~\ref{fig:Example4} shows that
$\wilde\eta_h$ and $\osc_h$ converge at a better rate. In Figure~\ref{fig:Example4Eta}, we plot only $\eta_h$ for uniform as well
as adaptive mesh refinement as well as the observed convergence rates. In Figure~\ref{fig:Example4AdapMesh}
we plot intermediate adaptive meshes with $99$ respectively $2779$ number of elements.

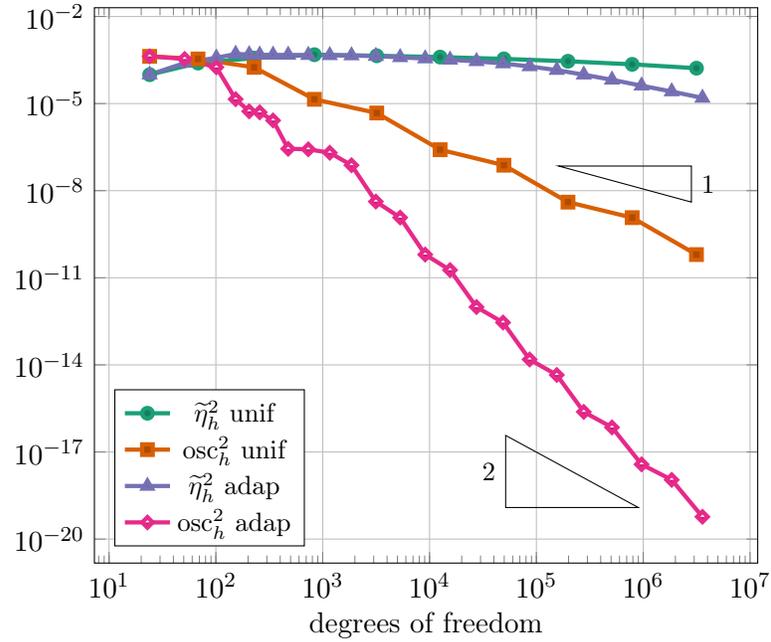
\begin{figure}
  \begin{center}
      \begin{tikzpicture}
\begin{loglogaxis}[
width=0.65\textwidth,
cycle list/Dark2-6,
cycle multiindex* list={
mark list*\nextlist
Dark2-6\nextlist},
every axis plot/.append style={ultra thick},
xlabel={degrees of freedom},
grid=major,
legend entries={\small $\widetilde \eta_h^2$ unif, \small $\osc_h^2$ unif,
\small $\widetilde \eta_h^2$ adap, \small $\osc_h^2$ adap},
legend pos=south west,
]
\addplot table [x=ndofs,y=est] {figures/Example4Unif.dat};
\addplot table [x=ndofs,y=osc] {figures/Example4Unif.dat};
\addplot table [x=ndofs,y=est] {figures/Example4Adap.dat};
\addplot table [x=ndofs,y=osc] {figures/Example4Adap.dat};

\logLogSlopeTriangle{0.9}{0.2}{0.65}{1}{black}{{\small $1$}};
\logLogSlopeTriangleBelow{0.82}{0.2}{0.1}{2}{black}{{\small $2$}};
\end{loglogaxis}
\end{tikzpicture}
  \end{center}
  \caption{Uniform and adaptive mesh-refinement for Experiment 4.}
  \label{fig:Example4}
\end{figure}
\begin{figure}
  \begin{center}
      \begin{tikzpicture}
\begin{loglogaxis}[
width=0.65\textwidth,
cycle list/Dark2-6,
cycle multiindex* list={
mark list*\nextlist
Dark2-6\nextlist},
every axis plot/.append style={ultra thick},
xlabel={degrees of freedom},
grid=major,
legend entries={\small $\widetilde \eta_h^2$ unif, \small $\widetilde \eta_h^2$ adap},
legend pos=south west,
]
\addplot table [x=ndofs,y=est] {figures/Example4Unif.dat};
\addplot table [x=ndofs,y=est] {figures/Example4Adap.dat};

\logLogSlopeTriangle{0.9}{0.2}{0.75}{0.16}{black}{{\small $0.16$}};
\logLogSlopeTriangleBelow{0.82}{0.2}{0.1}{0.75}{black}{{\small $0.75$}};
\end{loglogaxis}
\end{tikzpicture}
  \end{center}
  \caption{Uniform and adaptive mesh-refinement for Experiment 4, only $\eta_h$.}
  \label{fig:Example4Eta}
\end{figure}
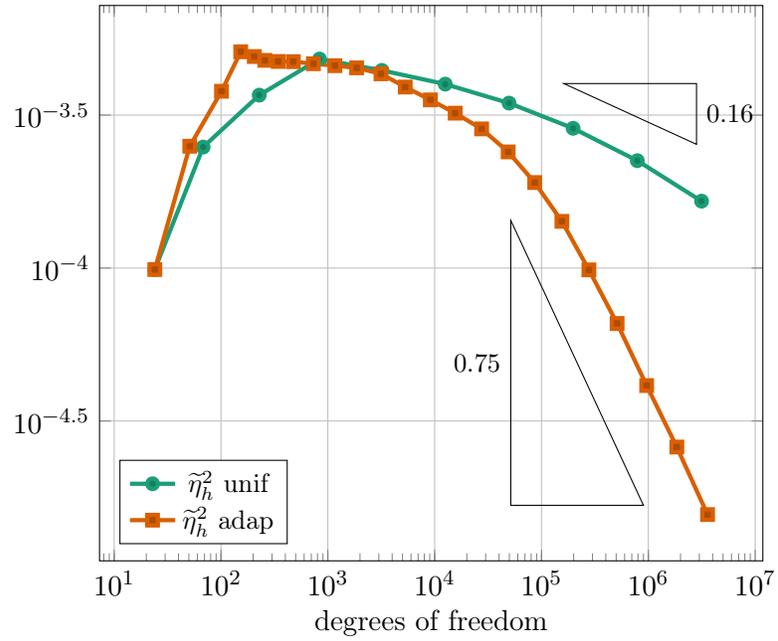
\begin{figure}
  \begin{center}
      \begin{tikzpicture}
\begin{axis}[
width=0.49\textwidth,
    axis equal,
]

\addplot[patch,color=white,
faceted color = black, line width = 0.5pt,
patch table ={figures/Example4_el1.dat}] file{figures/Example4_co1.dat};
\end{axis}
\end{tikzpicture}
      \begin{tikzpicture}
\begin{axis}[
width=0.49\textwidth,
    axis equal,
]

\addplot[patch,color=white,
faceted color = black, line width = 0.5pt,
patch table ={figures/Example4_el2.dat}] file{figures/Example4_co2.dat};
\end{axis}
\end{tikzpicture}
  \end{center}
  \caption{Adaptive meshes for Experiment 3 with $99$ resp $2779$ elements.}
  \label{fig:Example4AdapMesh}
\end{figure}
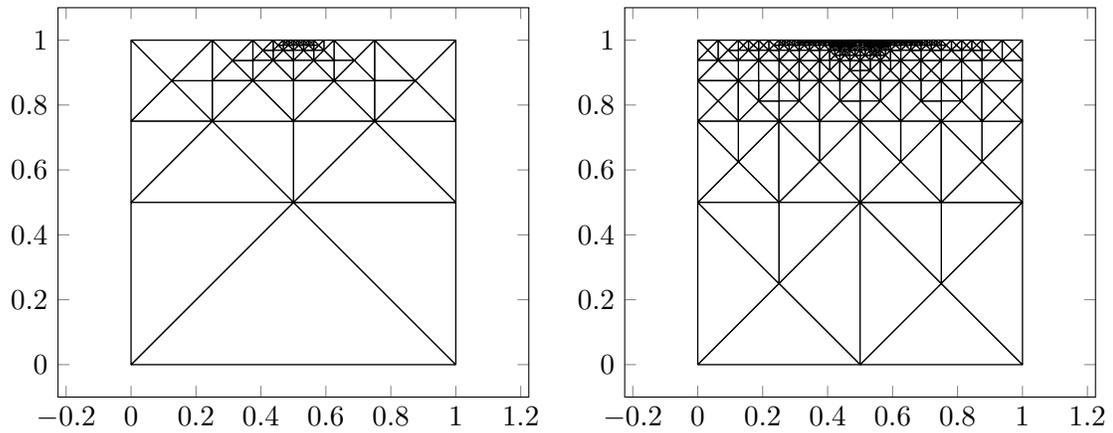
\bibliographystyle{abbrv}
\bibliography{literature}
\end{document}